\documentclass[oneside,english]{amsart}
\usepackage[T1]{fontenc}
\usepackage[latin9]{inputenc}
\usepackage{amsbsy}
\usepackage{amstext}
\usepackage{amsthm}
\usepackage{amssymb}
\usepackage{stmaryrd}

\makeatletter
\numberwithin{equation}{section}
\numberwithin{figure}{section}
\theoremstyle{plain}
\newtheorem{thm}{\protect\theoremname}[section]
\theoremstyle{definition}
\newtheorem{defn}[thm]{\protect\definitionname}
\theoremstyle{remark}
\newtheorem*{rem*}{\protect\remarkname}
\theoremstyle{plain}
\newtheorem{fact}[thm]{\protect\factname}
\theoremstyle{plain}
\newtheorem{lem}[thm]{\protect\lemmaname}
\theoremstyle{plain}
\newtheorem{cor}[thm]{\protect\corollaryname}
\theoremstyle{definition}
\newtheorem{example}[thm]{\protect\examplename}
\newcommand{\sub}{\subseteq}

\makeatother

\usepackage{babel}
\providecommand{\corollaryname}{Corollary}
\providecommand{\definitionname}{Definition}
\providecommand{\examplename}{Example}
\providecommand{\factname}{Fact}
\providecommand{\lemmaname}{Lemma}
\providecommand{\remarkname}{Remark}
\providecommand{\theoremname}{Theorem}

\begin{document}

\title{Definable One-Dimensional Topologies in O-minimal Structures}

\author{Ya\textquoteright acov Peterzil}

\address{Department of Mathematics, University of Haifa, Haifa, Israel}

\email{kobi@math.haifa.ac.il}

\author{Ayala Rosel}

\address{Department of Mathematics, University of Haifa, Haifa, Israel}

\email{ayalarosel@gmail.com}
\begin{abstract}
We consider definable topological spaces of dimension one in o-minimal
structures, and state several equivalent conditions for when such
a topological space $\left(X,\tau\right)$ is definably homeomorphic
to an affine definable space (namely, a definable subset of $M^{n}$
with the induced subspace topology). One of the main results says
that it is sufficient for $X$ to be regular and decompose into finitely
many definably connected components.
\end{abstract}

\maketitle

\section{Introduction}

The goal of this article is to study definable one-dimensional Hausdorff
topologies in o-minimal structures, and  to understand when they are definably homeomorphic
to a definable set in some\emph{ $M^{n}$ }with its affine topology
(namely, the induced subspace topology from \emph{$M^{n}$}).

When we say that $\tau$ is a definable topology on a definable set
$X$, we mean that $\tau$ has a basis which is definable in the language
of the underlying o-minimal structure.\\

Our main theorem consists of several equivalent conditions to $\left(X,\tau\right)$
being definably homeomorphic to a definable set with its affine topology.
It is a combination of Theorem \ref{thm: TFAE affine} and Theorem
\ref{thm: Condition =0000235 affine}:\\
\\
\textbf{Main theorem.}\emph{ Let $\mathcal{M}$ be an o-minimal expansion
of an ordered group. Let $X\subseteq M^{n}$ be a definable bounded set
with $\dim X=1$, and let $\tau$ be a definable Hausdorff topology
on $X$. Then the following are equivalent: }
\begin{enumerate}

\item \emph{$\left(X,\tau\right)$ is definably homeomorphic to a definable
subset of $M^{k}$ for some $k$, with its affine topology. }

\item \emph{There is a finite set $G\subseteq X$ such that every $\tau$-open
subset of $X\setminus G$ is open with respect to the affine topology
on $X\setminus G$. }

\item \emph{Every definable subset of $X$ has finitely many definably connected
components, with respect to $\tau$.}
\item \emph{$\tau$ is regular and $X$ has finitely many definably connected
components, with respect to $\tau$.}

\end{enumerate}

\noindent{\em Note: If $\mathcal M$ expands a real closed field then every definable set is in definable bijection with a bounded set, so 
 the assumption that $X$ is bounded could  be omitted.}
\vspace{.1cm}
 
We mention here a theorem of Erik Walsberg, which says that a definable
metric space in an o-minimal expansion of a real closed field is definably homeomorphic to
a definable set equipped with its affine topology if and only if  it does not
contain any infinite definable discrete set. This theorem can be found
in \cite{Walsberg}, and we shall phrase it more precisely later on.
Inspired by this work, we study general definable topological spaces,
but restrict our attention to dimension $1$.  \\

{\em The results of this article were part of the M.Sc. thesis of the second author
at the University of Haifa. 
After the submission of the thesis we learned that Pablo
Andujar Guerrero, Margaret Thomas and Erik Walsberg are working, independently, on
similar questions.  Finally, we thank the anonymous referee for some useful suggestions.}

\section{Basic definitions }

Below we take ``definable'' to mean ``definable with parameters''.
\begin{defn}
Let $\mathcal{M}=\left(M;\ldots\right)$ be a first-order structure
of a fixed language $\mathcal{L}$.  We say that {\em a collection $\mathcal B$ of subsets of $M^n$ is  definable (in $\mathcal M$)} if there exists an $\mathcal{L}$-formula $\varphi(\bar{x},\bar{y})$ such that $$\mathcal B=\{\varphi(M^n,\bar{b}):\bar{b}\in M^m\}.$$

Let $X\subseteq M^{n}$ be a definable
set. If $\mathcal{B}$ as above forms a basis for a topology $\tau$ on $X$, then we
say that $\tau$ is a \emph{definable topology} on $X$. (This is
the third possibility of considering topological structures from a
model-theoretic point of view due to Pillay's \cite{Pillay},
page 764, where it is named a \emph{first-order topological structure}.)

Note that a basis for the neighborhoods of $\bar{a}\in X$ is given
by
\[
\mathcal{B}_{\bar{a}}=\left\{ \varphi\left(X,\bar{b}\right):\models \varphi\left(\bar{a},\bar{b}\right),\bar{b}\in M^m\right\} =\left\{ U\in\mathcal{B}:\bar{a}\in U\right\} .
\]
\end{defn}

From now on, all topological operations, like closure or interior,
are taken with respect to the underlying topology $\tau$, unless
otherwise stated. The closure and interior of a subset $Z\subseteq X$
is denoted by $cl\left(Z\right)$ and $int\left(Z\right)$, respectively.
\\

This article investigates definable topologies in o-minimal structures.
We fix $\mathcal{M}=\left(M;<,\ldots\right)$ o-minimal and list some
examples of definable topologies in $\mathcal{M}$. All the following
examples can be defined on $X=M^{n}$, unless otherwise stated. We
can then also consider the induced topology $\tau|_{Y}$ for a definable
set $Y\subseteq X$.
\begin{enumerate}
\item The order topology on $M$, which we denote by $\tau^{<}$.
\item The affine topology $\tau^{af}$ which is the product topology with
respect to $\tau^{<}$.
\item In \cite{Dries}, the notion of a definable space is introduced, based
on a finite atlas where each chart is modeled on a definable subset
of some $M^{n}$, with its induced affine topology. It is easy to
verify that the associated topology is definable.
\item The discrete topology, which we denote by $\tau^{iso}$.
\item The left-closed topology on $M$, with the left closed-intervals as
basic open (also closed) sets. We denote this topology by $\tau^{^{\left[\,\,\right)}}$.
\item Every definable linear ordering $\prec$ on a definable set $X\subseteq M^{n}$
gives rise to a definable topology on $X$, namely the order topology
with respect to $\prec\,$. In \cite{Onshuus_Steinhorn}, Alf Onshuus
and Charles Steinhorn study such definable linear ordering in o-minimal
structures, and show that they are \textquotedblleft piecewise lexicographic\textquotedblright .
\item A definable metric topology $\tau^{d}$:
Let $\mathcal{M}=\left(M;<,+,\cdot,\ldots\right)$ be an o-minimal
expansion of a real closed field. A \emph{definable metric} on $X\subseteq M^{n}$
is a definable function $d:X^{2}\rightarrow M_{+}$ such that for
all $x,y,z\in X$: 1.~$d(x,y)=0\Leftrightarrow x=y.$ 2.~$d(x,y)=d(y,x)$.
3.~$d(x,z)\leq d(x,y)+d(y,z)$. The topology $\tau^{d}$ on $X$
is the topology whose basis is the collection of open balls with respect
to $d$.
\end{enumerate}
The following theorem is due to Erik Walsberg, \cite{Walsberg}:
\begin{thm}
Let $(X,d)$ be a definable metric space in an o-minimal expansion
of a real closed field. The following are equivalent:

(1) $(X,\tau^{d})$ is definably homeomorphic to a definable set with
it affine topology.

(2) There is no infinite definable set $A\subseteq X$ such that $(A,d)$
is discrete.
\end{thm}

\begin{rem*}
We note that the analogous result fails for definable topologies in
o-minimal structures. Indeed, consider the topology $\tau^{^{\left[\,\,\right)}}$
on $M$. There is no infinite definable set $A\subseteq X$ such that
$(A,\tau|_{A})$ is discrete, and yet it is not definably homeomorphic
to any definable set with the induced affine topology (e.g. by Theorem
\ref{thm: TFAE affine}).
\end{rem*}
We fix a definable set $X\subseteq M^{n}$ and a definable topology
$\tau$ on $X$ with a definable basis $\mathcal{B}$, and proceed
with some more definitions.
\begin{defn}
$\mathcal{F}\subseteq\mathcal{P}(X)$ is a \emph{filtered collection} if for every $B_{1},B_{2}\in\mathcal{F}$ there exists $B_{3}\in\mathcal{F}$
such that $B_{3}\subseteq B_{1}\cap B_{2}$.
\end{defn}

An example of a filtered collection is a basis for the neighborhoods
of each $\bar{a}\in X$.\\

The following definition was given by Will Johnson in \cite{Johnson}
(see \cite{Peterzil_Steinhorn} for an earlier definition in the o-minimal
setting):
\begin{defn}
$\left(X,\tau\right)$ is \emph{definably compact} if every definable
filtered collection of closed non-empty subsets of X has non-empty
intersection.
\end{defn}

The next lemma is proved in \cite{Johnson} (Corollary 1.11 and the subsequent
paragraph):
\begin{fact}
\label{lem: J} If $\mathcal{M}$ is o-minimal and $X\subseteq M^{n}$
is definable, then \textup{$X$ is }definably compact with respect
to the induced affine topology if and only if  \textup{$X$ is }$\tau^{af}$-closed
and bounded. In particular, every definable filtered collection of
$\tau^{af}$-closed and bounded non-empty subsets, has non-empty intersection.
\end{fact}

Thus, the definition above is equivalent in the o-minimal setting, for the affine topology, to the one in \cite{Peterzil_Steinhorn}.

\medskip{}

\textbf{In our context, whenever we say that
a set is definably compact, we mean that it is definably compact with
respect to the induced affine topology. }\\

We recall:
\begin{defn}
\label{def: regular} $\left(X,\tau\right)$ is \emph{regular} if
for every $\bar{a}\in X$ and open $U\subseteq X$ with $\bar{a}\in U$
there is an open $W\subseteq X$ with $\bar{a}\in W$ such that $cl\left(W\right)\subseteq U$.
\end{defn}

\begin{fact}
For any topology $\tau$ on $X$, $\left(X,\tau\right)$ is
regular  if and only if  for every basis $\mathcal{B}$
for~$\tau$, for every point $\bar{a}\in X$ and open basic neighborhood
$U\in\mathcal{B}$ of $\bar{a}$, there is an open basic neighborhood
$W\in\mathcal{B}$ of $\bar{a}$ such that $cl\left(W\right)\subseteq U$.
\end{fact}

It follows that for a definable topology\emph{ $\tau$ on $X$, $\left(X,\tau\right)$}
is regular if and only if  it is definably regular, namely, in Definition \ref{def: regular}
we may consider only definable $U,W$.\\

We continue with a new definition:
\begin{defn}
Let $T,S\subseteq\mathcal{P}\left(X\right)$ be two definable families
of sets. We write $T\preceq S$ if for each $U\in T$ there is $V\in S$
such that $V\subseteq U$. If both $T\preceq S$ and $S\preceq T$
take place, we say that \emph{the families $T$ and $S$ are equivalent},
and write $T\sim S$. By $T\precneqq S$ we mean that $T\preceq S$
and $T\nsim S$.

In particular, let $\tau$ and $\eta$ be definable topologies on
$X$ with definable bases $\mathcal{B}^{\tau}$ and $\mathcal{B}^{\eta}$,
respectively. The topology $\eta$ is finer than the topology $\tau$
if and only if for every $\bar{a}\in X$, $\mathcal{B}_{\bar{a}}^{\tau}\preceq\mathcal{B}_{\bar{a}}^{\eta}$.
If we have both $\mathcal{B}_{\bar{a}}^{\tau}\preceq\mathcal{B}_{\bar{a}}^{\eta}$
and $\mathcal{B}_{\bar{a}}^{\eta}\preceq\mathcal{B}_{\bar{a}}^{\tau}$,
we say that \emph{the bases $\mathcal{B}_{\bar{a}}^{\tau}$ and $\mathcal{B}_{\bar{a}}^{\eta}$
are equivalent}, and write $\mathcal{B}_{\bar{a}}^{\tau}\sim\mathcal{B}_{\bar{a}}^{\eta}$.
It follows that $\tau=\eta$ if and only if  for all $\bar{a}\in X$, $\mathcal{B}_{\bar{a}}^{\tau}\sim\mathcal{B}_{\bar{a}}^{\eta}$.
\end{defn}

\begin{lem}
\label{Let--formula} Let $\psi\left(\bar{y}\right)$, $|\bar{y}|=m$,
be an $\mathcal{L}$-formula. Let $\bar{a}\in X$, and assume
that for each neighborhood $U\in\tau$ of \textup{$\bar{a}$ }\textup{\emph{there
exists a}} neighborhood $U_{\bar{b}}\in\mathcal{B}_{\bar{a}}$ of
\textup{$\bar{a}$}\textup{\emph{ }}such that $\mathcal{M}\models\psi\left(\bar{b}\right)$\textup{
 }and\textup{\emph{ }}$U_{\bar{b}}\subseteq U$\textup{. }\textup{\emph{Then}}
$\left\{ U_{\bar{b}}:\mathcal{M}\models\psi\left(\bar{b}\right)\right\} \sim\mathcal{B}_{\bar{a}}$.
\end{lem}

\begin{proof}
Clearly, $\left\{ U_{\bar{b}}:\mathcal{M}\models\psi\left(\bar{b}\right)\right\} \subseteq\left\{ U_{\bar{b}}:\bar{b}\in M^m\right\} =\mathcal{B}_{\bar{a}}$.
By the assumption of the lemma, for each $U\in\mathcal{B}_{\bar{a}}$
there is $U_{\bar{b}}\in\left\{ U_{\bar{b}}:\mathcal{M}\models\psi\left(\bar{b}\right)\right\} $
such that\emph{ }$U_{\bar{b}}\subseteq U$. Thus $\left\{ U_{\bar{b}}:\mathcal{M}\models\psi\left(\bar{b}\right)\right\} \sim\mathcal{B}_{\bar{a}}$.
\end{proof}
Informally, Lemma \ref{Let--formula} says that if a definable property
holds for arbitrary small neighborhoods of $\bar{a}$, then one can
pick a definable basis for the neighborhoods of $\bar{a}$ such that this property
holds for all of its sets.

As an immediate corollary we obtain:
\begin{lem}
\label{Given-a-property} Let $\bar{a}\in X$, and assume that for
every $U_{\bar{b}}\in\mathcal{B}_{\bar{a}}$, $\mathcal{M}\models\left(\psi_{1}\left(\bar{b}\right)\lor\psi_{2}\left(\bar{b}\right)\right)$.
For $i=1,2$, denote $\mathcal{B}_{\bar{a}}^{i}=\left\{ U_{\bar{b}}\in\mathcal{B}_{\bar{a}}:\psi_{i}\left(\bar{b}\right)\right\} $.
Then either $\mathcal{B}_{\bar{a}}^{1}\mathcal{\sim B}_{\bar{a}}$
or $\mathcal{B}_{\bar{a}}^{2}\mathcal{\sim B}_{\bar{a}}$.
\end{lem}

We end this part with some definitions that we use later
on:
\begin{defn}
$\left(X,\tau\right)$ is \emph{definably connected} if there are
no definable non-empty open sets $U,W$ such that $U\cap W=\emptyset$
and $U\cup W=X$. Equivalently, $X$ does not contain any definable
proper non-empty clopen subset. A definable $Y\subseteq X$ is \emph{definably
connected} (with respect to $\tau$) if the space $\left(Y,\tau|_{Y}\right)$
is definably connected\emph{.}
\end{defn}

\begin{defn}
Let $\tau$ be a definable topology on $X$. A \textbf{definable},
maximal definably connected subset of $X$ is called a \emph{definably
connected component} of $X$.

If $X$ can be decomposed into finitely many definably connected components,
then we say that $\left(X,\tau\right)$ has finitely many definably
connected components.

The space $\left(X,\tau\right)$ is called \emph{totally definably
disconnected} if its only definably connected subsets are singletons
and $\emptyset$. A definable \emph{subset} $A\subseteq X$ is called
\emph{totally definably disconnected} if it is so with respect to
the subspace topology.
\end{defn}

Note that each definably connected component is a clopen set (since
the closure of a definably connected set is itself definably connected).
\begin{rem*}
We do not know in general the answer to the following question: Given
a definable topology $\tau$ on $X$ and $\bar{a}\in X$, is the union
of all (definable) definably connected subsets of $X$ which contain
$\bar{a}$, a definable set itself?

\end{rem*}

\section{Towards the main results}

From now on we assume that \\
\textbf{$\boldsymbol{\mathcal{M}=\left(M;<,+,\ldots\right)}$
is an o-minimal expansion of an ordered group, $\boldsymbol{X\subseteq M^{n}}$
is a definable one-dimensional bounded set and $\boldsymbol{\tau}$
is an definable Hausdorff topology on $\boldsymbol{X}$. For simplicity,
assume that $\boldsymbol{X}$ and $\boldsymbol{\tau}$ are definable
over $\boldsymbol{\emptyset}$. }

Whenever we mention a topology without pointing out which one, we
are referring to the topology $\tau$. Thus, whenever we write $\mathcal B_{\bar{a}}$ we refer to the basis
of neighborhoods with respect to  $\tau$. Having said that, we refer to various types of $<$-intervals by their usual names. E.g, the term ``open-interval'' always refers
to an interval of the form $(a,b)$, for $a<b$ in $M$.

By o-minimality, $X$ is a finite union of $0$-cells
and bounded  $1$-cells in $\mathcal{M}$. Hence, by applying an appropriate
$\mathcal{L}$-definable bijection 
we can assume that $X$ is a bounded subset of $M$:\textbf{
\[
\boldsymbol{X=\left(s_{1},t_{1}\right)\sqcup\ldots\sqcup\left(s_{l},t_{l}\right)\sqcup F},
\]
 where $\boldsymbol{F}$ is a finite set of points, each $\boldsymbol{s_{i},t_{i}}$
is in $\boldsymbol{M}$, and $\boldsymbol{s_{i}\neq t_{j}}$ for all
$\boldsymbol{1\leq i\neq j\leq l}$. }

Let $\mathcal{B}^{af}=\left\{ \left(b_{1},b_{2}\right):b_{1},b_{2}\in M\right\} $
be the standard definable basis for $\left(M,\tau^{af}\right)$. A
basis for the $\tau^{af}$-neighborhoods in $X$
of $a\in X$ is
\[
\mathcal{B}_{a}^{af}\left(X\right):=\left\{ \left(b_{1},b_{2}\right)\cap X:a\in\left(b_{1},b_{2}\right)\right\} .
\]
 To simplify notation, we write $\mathcal{B}_{a}^{af}$ instead of
$\mathcal{B}_{a}^{af}\left(X\right)$.\\

We need a few more definitions. Below the term ``locally'' refers to the affine topology.
\begin{defn}
 We say that the point $a\in X$ is \emph{locally isolated }if there
are $U\in\mathcal{B}_{a}$ and an open-interval $I\ni a$ such that
$U\cap I=\left\{ a\right\} $.
\end{defn}

\begin{defn}
We say that the point $a\in X$ is \emph{locally right-closed }if
for every small enough $U\in\mathcal{B}_{a}$ there exists an open-interval
$I_{U}\ni a$ and a point $a'\in X$, $a'<a$, such that $U\cap I_{U}=\left(a',a\right]$.
A \emph{locally left-closed }point is defined similarly.
\end{defn}

\begin{defn}
We say that the point $a\in X$ is \emph{locally Euclidean }if for
every small enough $U\in\mathcal{B}_{a}$ there exists an open-interval
$I_{U}\ni a$ and two points $a',a''\in X$, $a'<a<a''$, such that
$U\cap I_{U}=\left(a',a''\right)$.
\end{defn}

Here is an easy observation:
\begin{lem}
\emph{\label{lem: locally iso or open or half} }For every\emph{ }\textup{\emph{$a\in X$, exactly one of the following
holds:}}

\textup{\emph{(1) $a$ is}} locally isolated.

\textup{\emph{(2) $a$ is}} locally right-closed.

\textup{\emph{(3) $a$ is}}\emph{ }locally left-closed.

\textup{\emph{(4) $a$ is}} locally Euclidean.\textup{ }
\end{lem}

\begin{proof}
Fix $a\in X$. By o-minimality, every definable subset $U$ containing
$a$ is a finite union of points and intervals. This means that there
exists an open-interval $I\ni a$ such that $I\cap U$ is either $\left\{ a\right\} $
or a half-closed-interval, or half-open-interval,  or an open-interval. Thus, $U$ is from
one of the above four types and it is easy to see that each of those
is a definable property of $U$. The result follows by Lemma \ref{Given-a-property}.
\end{proof}
Notice that if $\mathcal{B}_{a}\mathcal{\npreceq B}_{a}^{af}$ then
$a$ is not locally Euclidean. Hence we have:
\begin{cor}
\label{cor: a locally iso=00005Cleft=00005Cright} For every $a\in X$,
if $\mathcal{B}_{a}\mathcal{\npreceq B}_{a}^{af}$ then $a$ is either
locally isolated or locally right-closed or locally left-closed\emph{.}
\end{cor}

\subsection{\label{subsec: shadows-points}The shadows of a point}

 For a set $U\subseteq X$, by $cl^{af}\left(U\right)$
we mean the $\tau^{af}$-closure of $U$ \textbf{in }$\boldsymbol{M}$,
even if $cl^{af}\left(U\right)\nsubseteq X$. The definition below, which plays a key role in our analysis of the topology, 
makes sense for every definable topology on $X$ (independently of $\dim X$). 
\begin{defn}
We define the\emph{ set of shadows} of a point $a\in X$ to be
\[
\boldsymbol{S}\left(a\right):=\bigcap_{U\in\mathcal{B}_{a}}cl^{af}\left(U\right).
\]
 We call a point in $\boldsymbol{S}\left(a\right)$ a \emph{shadow}
of $a$.
\end{defn}

It is not hard to see that $\boldsymbol{S}\left(a\right)$ does not
depend on the specific choice of a basis. Intuitively, as the example below shows, $\boldsymbol{S}\left(a\right)$ is the set of all points in $M^n$
which are ``glued'' to $a$ by the topology $\tau$.
\begin{example}
\label{exa: =00005Cinfty} Let $\tau$ be the topology on $X=\left(s,t\right)\subseteq\mathbb{R}$
that is homeomorphic to the figure $\infty$, as follows: Fix $a\in\left(s,t\right)$,
and let
\[
\mathcal{B}^{\tau}:=\left\{ \left(r_{1},r_{2}\right)\subseteq\left(s,t\right):r_{1}<r_{2}\leq a\text{ or }a\leq r_{1}<r_{2}\right\} \cup
\]
\[
\left\{ \left(s,r_{1}\right)\cup\left(r_{2},r_{3}\right)\cup\left(r_{4},t\right):r_{1}\leq r_{2}<a<r_{3}\leq r_{4}\right\} .
\]

The point $a$ corresponds to the middle point of $\infty$, where
$s$ and $t$ are attached to $a$. The figure $\infty$ shape is
formed by closing the two sides of $\left(s,t\right)$ to the point
$a$. Notice that $a$ is the only point such that $\mathcal{B}_{a}\nsim\mathcal{B}_{a}^{af}$,
and $\boldsymbol{S}\left(a\right)=\left\{ a,s,t\right\} $.
\end{example}

\begin{example}
Let $M^{2}\supseteq X=\left(0,1\right)\times\left\{ 1,2\right\} $.
Let $\prec$ be the lexicographic order on $X$, and $\tau^{\prec}$
be the associated topology on $X$. One may check that for every $c\in\left(0,1\right)$,
\[
\boldsymbol{S}\left(\left\langle c,1\right\rangle \right)=\boldsymbol{S}\left(\left\langle c,2\right\rangle \right)=\left\{ \left\langle c,1\right\rangle ,\left\langle c,2\right\rangle \right\} .
\]
\end{example}

The following is immediate:
\begin{fact}
For every $a\in X$,

(1) $a\in\boldsymbol{S}\left(a\right)$.

(2) If $\tau$ is the affine topology on $X$ then $\boldsymbol{S}\left(a\right)=\left\{ a\right\} $.
\end{fact}

We still assume below that $X\subseteq M$ is a definable one dimensional
bounded set.
\begin{lem}
\label{lem:GeneralProperty2 of Y_a} For every $a\in X$, the following
are equivalent:

(1) $\boldsymbol{S}\left(a\right)=\left\{ a\right\} $.

(2) $\mathcal{B}_{a}^{af}\preceq\mathcal{B}_{a}$.

(3) $\mathcal{B}_{a}\sim\mathcal{B}_{a}^{iso}$ or $\mathcal{B}_{a}\sim\mathcal{B}_{a}^{^{\left(\,\,\right]}}$
or $\mathcal{B}_{a}\sim\mathcal{B}_{a}^{^{\left[\,\,\right)}}$ or
$\mathcal{B}_{a}\sim\mathcal{B}_{a}^{af}$ (where $\mathcal{B}_{a}^{af}=\mathcal{B}_{a}^{af}\left(X\right)$).
\end{lem}

\begin{proof}
$(1)\Rightarrow(2)$: Assume $\mathcal{B}_{a}^{af}\npreceq\mathcal{B}_{a}$.
Then there exists an interval $I\in\mathcal{B}_{a}^{af}$ such that for every $U\in\mathcal{B}_{a}$,
$U\nsubseteq I$. Since $X$ is bounded, every $cl^{af}\left(U\right)$
is a nonempty $\tau^{af}$-closed and bounded set, and so is every
$cl^{af}\left(U\setminus I\right)$. Thus, $\left\{ cl^{af}\left(U\setminus I\right):U\in\mathcal{B}_{a}\right\} $
is a definable filtered collection of $\tau^{af}$-closed and bounded
non-empty sets, so by Lemma \ref{lem: J}, its intersection is non-empty.

Since $a\in I$, it does not belong to this intersection. Therefore, $\bigcap_{U\in\mathcal{B}_{a}}cl^{af}\left(U\setminus I\right)$
must contain an element different than $a$. Finally, note that
\[
\boldsymbol{S}\left(a\right)=\bigcap_{U\in\mathcal{B}_{a}}cl^{af}\left(U\right)\supseteq\bigcap_{U\in\mathcal{B}_{a}}cl^{af}\left(U\setminus I\right),
\]
 hence $\boldsymbol{S}\left(a\right)$ contains an element other than
$a$. That is, $\left\{ a\right\} \subsetneq\boldsymbol{S}\left(a\right)$.

$(2)\Rightarrow(3)$: We assume that $a$ satisfies
one of (1)-(4) from Lemma \ref{lem: locally iso or open or half}.
Assume for example that it is locally right-closed, namely  for every sufficiently small 
$U\in\mathcal{B}_{a}$ there exists
an open-interval $I\ni a$ such that $I\cap U=\left(a',a\right]$.
Note that this assumption implies that $\mathcal{B}_{a}\preceq\mathcal{B}_{a}^{^{\left(\,\,\right]}}$.
We show that $\mathcal{B}_{a}\sim\mathcal{B}_{a}^{^{\left(\,\,\right]}}$
(the other cases are treated similarly).

We need to prove $\mathcal{B}_{a}^{^{\left(\,\,\right]}}\preceq\mathcal{B}_{a}$.
Fix $\left(a-\epsilon,a\right]\in\mathcal{B}_{a}^{^{\left(\,\,\right]}}$,
and we show that for some $W\in\mathcal{B}_{a}$, $W\subseteq\left(a-\epsilon,a\right]$.
Consider the interval $\left(a-\epsilon,a+\epsilon\right)$. By the
assumption $\mathcal{B}_{a}^{af}\preceq\mathcal{B}_{a}$, there exists
$U\in\mathcal{B}_{a}$ such that $U\subseteq\left(a-\epsilon,a+\epsilon\right)$.
By our assumption on $a$, there must be an open-interval $I\ni a$
such that $I\cap U=\left(a',a\right]$. We can take $I$ small enough,
and assume that $I=\left(a-\delta,a+\delta\right)\subseteq\left(a-\epsilon,a+\epsilon\right)$
with $a-\delta\leq a'<a$. Once again by the assumption $\mathcal{B}_{a}^{af}\preceq\mathcal{B}_{a}$,
there exists $V\in\mathcal{B}_{a}$ such that $V\subseteq\left(a-\delta,a+\delta\right)$.

Finally, there is $W\in\mathcal{B}_{a}$ such that $W\subseteq V\cap U$.
Thus we have
\[
W\subseteq V\cap U\subseteq\left(a-\delta,a+\delta\right)\cap U=\left(a',a\right]\subseteq\left(a-\epsilon,a\right].
\]
 Therefore $\mathcal{B}_{a}^{^{\left(\,\,\right]}}\preceq\mathcal{B}_{a}$,
and thus by our assumptions we have $\mathcal{B}_{a}\sim\mathcal{B}_{a}^{^{\left(\,\,\right]}}$.

$(3)\Rightarrow(1)$: Direct verification.
\end{proof}
\begin{lem}
\label{lem: B_a < B_a^M --> Y_a =00005Cneq =00007Ba=00007D} For every
$a\in X$, if $\mathcal{B}_{a}\precneqq\mathcal{B}_{a}^{af}$ then
$\boldsymbol{S}\left(a\right)\neq\left\{ a\right\} $.
\end{lem}

\begin{proof}
If $\mathcal{B}_{a}\precneqq\mathcal{B}_{a}^{af}$, then it follows
that $\mathcal{B}_{a}^{af}\npreceq\mathcal{B}_{a}$, and thus by Lemma
\ref{lem:GeneralProperty2 of Y_a} we have $\boldsymbol{S}\left(a\right)\neq\left\{ a\right\} $.
\end{proof}
\begin{lem}
\label{lem: a locally isolated implies Y_a > =00007Ba=00007D} For
every $a\in X$, if $a$ is locally isolated and not isolated then
$\boldsymbol{S}\left(a\right)\supsetneqq\left\{ a\right\} $.
\end{lem}

\begin{proof}
Since $a$ is not isolated , $\mathcal B_a\nsim \mathcal B_a^{iso}$. Since it is not 
locally isolated, the other possibilities of
Lemma \ref{lem:GeneralProperty2 of Y_a} (3) fail as well. Therefore, by Clause (1) of that lemma, $\boldsymbol{S}\left(a\right)\supsetneqq\left\{ a\right\} $.
\end{proof}
Lemmas \ref{lem: (*)} - \ref{lem:GeneralProperty1 of Y_a} can be
easily generalized for $X$ of arbitrary dimension $n$ (by considering
basic $\tau^{af}$-open sets of dimension $n$ and their closure instead
of open and closed-intervals).
\begin{lem}
\label{lem: (*)} For every $a,b\in X$, if $b\in\boldsymbol{S}\left(a\right)$
and $b\ne a$ then $\mathcal{B}_{b}\npreceq\mathcal{B}_{b}^{af}$.
\end{lem}

\begin{proof}
If $\mathcal{B}_{b}\preceq\mathcal{B}_{b}^{af}$, then every $U\in\mathcal{B}_{b}$
would contain an open interval $I_{U}\ni b$, hence $b$ could not
be separated from $a$, in contradiction to the fact that $\tau$
is Hausdorff.
\end{proof}
\begin{lem}
\label{lem:C1. b in Y_a if and only if  a in cl(I)} Let $a\in X$ and $b\in M$.
Then

$b\in\boldsymbol{S}\left(a\right)\iff$ For every open-interval $I\ni b$,
$a\in cl\left(I\cap X\right)$.
\end{lem}

\begin{proof}
Let $b\in\boldsymbol{S}\left(a\right)=\bigcap_{U\in\mathcal{B}_{a}}cl^{af}\left(U\right)$.
So for each $U\in\mathcal{B}_{a}$, $b\in cl^{af}\left(U\right)$.
That is, for every $U\in\mathcal{B}_{a}$ and every $V\in\mathcal{B}_{b}^{af}$,
we have $U\cap V\neq\emptyset$. Therefore, for every $U\in\mathcal{B}_{a}$
and every open-interval $I\ni b$, we have $U\cap I\neq\emptyset$.
 This exactly means that for every open-interval $I\ni b$, $a\in cl\left(I\right)$.

For the other direction, just follow from bottom to top: Assume that
for every open-interval $I\ni b$ we have $a\in cl\left(I\right)$.
That is, for every open-interval $I\ni b$, for every $U\in\mathcal{B}_{a}$,
we have $U\cap I\neq\emptyset$. This means that for every $U\in\mathcal{B}_{a}$,
$b\in cl^{af}\left(U\right)$. Since $\boldsymbol{S}\left(a\right)=\bigcap_{U\in\mathcal{B}_{a}}cl^{af}\left(U\right)$,
we get $b\in\boldsymbol{S}\left(a\right)$.
\end{proof}
\begin{lem}
\label{lem:newC cl(I)} Let $\left(c,d\right)\subseteq X$ be an open-interval.
Then
\[
\left\{ x\in X:\boldsymbol{S}\left(x\right)\cap\left(c,d\right)\neq\emptyset\right\} \subseteq cl\left(\left(c,d\right)\right)\subseteq\left\{ x\in X:\boldsymbol{S}\left(x\right)\cap\left[c,d\right]\neq\emptyset\right\} .
\]
\end{lem}

\begin{proof}
The first inclusion follows from direction $\Rightarrow$ of Lemma
\ref{lem:C1. b in Y_a if and only if  a in cl(I)}. For the second inclusion,
let $x_{0}\in cl\left(\left(c,d\right)\right)$. That is, for every
$U\in\mathcal{B}_{x_{0}}$ we have $U\cap\left(c,d\right)\neq\emptyset$.
Therefore, for every $U\in\mathcal{B}_{x_{0}}$ we also must have
$cl^{af}\left(U\right)\cap cl^{af}\left(\left(c,d\right)\right)\neq\emptyset$.
Since $cl^{af}\left(\left(c,d\right)\right)=\left[c,d\right]$, this
gives $\left(\bigcap_{U\in\mathcal{B}_{x_{0}}}cl^{af}\left(U\right)\right)\cap\left[c,d\right]\neq\emptyset$.
That is, $\boldsymbol{S}\left(x_{0}\right)\cap\left[c,d\right]\neq\emptyset$.
\end{proof}
\begin{lem}
\label{lem:C2. Y_a is finite } For every $a\in X$, $\boldsymbol{S}\left(a\right)$
is a finite set. Moreover, $|\boldsymbol{S}\left(a\right)|$ is uniformly
bounded, that is, there exists $k\in\mathbb{N}$ such that for all
$a\in M$, $|\boldsymbol{S}\left(a\right)|\leq k$.
\end{lem}

\begin{proof}
Fix $a\in X$. Since $\tau$ is Hausdorff, for every $x\in X$, if
$x\neq a$ then there is $U_{0}\in\mathcal{B}_{a}$ such that $x\notin U_{0}$,
hence if $x\in\boldsymbol{S}\left(a\right)$, then $x\in\left(cl^{af}\left(U_{0}\right)\setminus U_{0}\right)$.
Since $x\in\boldsymbol{S}\left(a\right)=\bigcap_{U\in\mathcal{B}_{a}}cl^{af}\left(U\right)$,
we deduce $x\in\bigcap_{U\in\mathcal{B}_{a}}\left(cl^{af}\left(U\right)\setminus U\right)$,
and thus we get $\boldsymbol{S}\left(a\right)\setminus\left\{ a\right\} \subseteq\bigcap_{U\in\mathcal{B}_{a}} cl^{af}\left(U\right)\setminus U$.

By general properties of the o-minimal topology and since $X\subseteq M$,
for every set $U\subseteq M$ we have that $cl^{af}\left(U\right)\setminus U$
is finite. So
\[
|\boldsymbol{S}\left(a\right)|\leq|cl^{af}\left(U_{0}\right)\setminus U_{0}|+1,
\]
where the $+1$ stands for $a$ itself. In particular, $\boldsymbol{S}\left(a\right)$
is a finite set.

Moreover, when $a$ varies over all elements of $M$ we have that
$\left\{ \boldsymbol{S}\left(a\right):a\in M\right\} $ is a definable
family, hence by o-minimality it has a uniform bound. That is, there
is $k\in\mathbb{N}$ such that for all $a\in M$, $|\boldsymbol{S}\left(a\right)|\leq k$.
\end{proof}
The generalization of Lemma \ref{lem:C2. Y_a is finite } above to arbitrary
dimension would just say that for every $x\in X$, $dim\left(\boldsymbol{S}\left(a\right)\right)<\dim X$.
\begin{lem}
\label{lem:GeneralProperty1 of Y_a} Denote $\boldsymbol{S}\left(a\right)=\left\{ a_{1},a_{2},\ldots a_{r}\right\} $.
Then for any open-intervals $I_{i}\ni a_{i}$, $1\leq i\leq r$, there
exists $U\in\mathcal{B}_{a}$ such that $U\subseteq\bigcup_{i=1}^{r}I_{i}$.
\end{lem}

\begin{proof}
Fix $I_{i}\ni a_{i},1\leq i\leq r$. Assume towards contradiction
that for every $U\in\mathcal{B}_{a}$, $U\nsubseteq\bigcup_{i=1}^{r}I_{i}$.
Therefore, $\left\{ cl^{af}\left(U\setminus\left(\bigcup_{i=1}^{r}I_{i}\right)\right):U\in\mathcal{B}_{a}\right\} $
is a definable filtered family of non-empty $\tau^{af}$-closed and
bounded sets. Thus, each set $cl^{af}\left(U\setminus\left(\bigcup_{i=1}^{r}I_{i}\right)\right)$
is definably compact, and therefore, their intersection is non-empty:
\[
\emptyset\neq\bigcap_{U\in\mathcal{B}_{a}}cl^{af}\left(U\setminus\left(\bigcup_{i=1}^{r}I_{i}\right)\right)\subseteq\bigcap_{U\in\mathcal{B}_{a}}cl^{af}\left(U\right)=\boldsymbol{S}\left(a\right)=\left\{ a_{1},a_{2},\ldots a_{r}\right\} .
\]
 But since $a_{i}\in I_{i}=int^{af}\left(I_{i}\right)$ for each $1\leq i\leq r$,
then
\[
a_{i}\notin\bigcap_{U\in\mathcal{B}_{a}}cl^{af}\left(U\setminus\left(\bigcup_{i=1}^{r}I_{i}\right)\right)
\]
 for each $1\leq i\leq r$, and this is a contradiction.
\end{proof}
As we noted earlier, each $b$ is in $\boldsymbol{S}\left(b\right)$.
Now we prove:
\begin{lem}
\label{lem:C3. 2pts a s.t. b in Y_a} For every $b\in M$, there are
at most two points $a\in X$ other than $b$, such that $b\in\boldsymbol{S}\left(a\right)$.
Moreover, if $b\in X$ is not locally isolated then there exists at most
one such $a$.
\end{lem}

\begin{proof}
If $b\in\boldsymbol{S}\left(a\right)$ then for every $U\in\mathcal{B}_{a}$
we have $b\in cl^{af}\left(U\right)$. Since $\mathcal{M}$ is o-minimal
there is an open-interval $I_{U}\subseteq U$ such that $b\in cl^{af}\left(I_{U}\right)$.
If $b\in X$ then there must be $U\in\mathcal{B}_{a}$ such that $b\notin U$,
hence $b$ is one of the end points of $I_{U}$. If $b\notin X$ then
also $b$ is an end point of $I_{U}$.

Assume that there are two distinct points $a_{1},a_{2}\in X$ different
than $b$, such that $b\in\boldsymbol{S}\left(a_{1}\right),\boldsymbol{S}\left(a_{2}\right)$.
Then, since $\tau$ is Hausdorff, there are disjoint $U_{1}\in\mathcal{B}_{a_{1}}$,
$U_{2}\in\mathcal{B}_{a_{2}}$ with $b\in cl^{af}\left(I_{U_{1}}\right)\cap cl^{af}\left(I_{U_{2}}\right)$.
Because $\tau$ is Hausdorff, $b$ is a left end point of one of these
intervals $I_{U_{i}}$ and a right end point of the other. For the
same reason, $b$ is locally isolated and there cannot be a third
point $a_{3},$ with $b\in\boldsymbol{S}\left(a_{3}\right)$.
\end{proof}

\subsection{\label{subsec: Shadows-generic-points}Shadows of generic points}

In this subsection we work in an elementary extension $\mathcal{N}=\left(N;<,\ldots\right)$
of $\mathcal{M}$ which is sufficiently saturated. Note that now we
have $\left(X,\tau\right)=\left(X\left(N\right),\tau\left(N\right)\right)$
(instead of $\left(X,\tau\right)=\left(X\left(M\right),\tau\left(M\right)\right)$
). It is easy to verify that $\tau\left(N\right)$ is still a Hausdorff
topology on $X\left(N\right)$. We assume that $X$ and $\tau$ are
definable over $\emptyset$. \\

We first recall the following known lemma:
\begin{lem}
\label{lem:Generic open interval} Let $a$ be a generic point in
$X\sub M^n$ over $\emptyset$. Let $U\subseteq M^n$ be a definable affine-open neighborhood of $a$, defined
over a set $A$ of parameters. Then there exists a definable affine-open neighborhood  of $a$, $W\subseteq U$, defined
over a set $B\supseteq A$, such that $dim\left(a/A\right)=dim\left(a/B\right)$.

In particular, let $a\in M$ be a generic point over $\emptyset$.
Then we can choose an arbitrary small interval $\left(a_{1},a_{2}\right)$,
$a_{1}<a<a_{2}$, such that $a$ is still generic over $\left\{ a_{1},a_{2}\right\} $.
\end{lem}

\begin{lem}
\label{lem: a generic implies b generic} If $a\in X$ and $b\in\boldsymbol{S}\left(a\right)$,
then $b$ is generic over $\emptyset$ if and only if $a$ is generic
over $\emptyset$.
\end{lem}

\begin{proof}
Since $\boldsymbol{S}\left(a\right)$ is an $a$-definable finite
set by Lemma \ref{lem:C2. Y_a is finite }, $b\in\boldsymbol{S}\left(a\right)$
implies that $b\in acl\left(a\right)$. By Lemma \ref{lem:C3. 2pts a s.t. b in Y_a},
there is a finite number of points $a'\in X$ such that $b\in\boldsymbol{S}\left(a'\right)$.
Since the set of all these points is definable over $b$, we also
have $a\in acl\left(b\right)$. It follows that $b$ is generic over
$\emptyset$ if and only if  $a$ is generic over $\emptyset$.
\end{proof}
\begin{lem}
\label{lem:Lemma Y_b in Y_a} Let $a\in X$ be a generic point over
$\emptyset$, and let $b\in\boldsymbol{S}\left(a\right)$. Then $\boldsymbol{S}\left(b\right)\subseteq\boldsymbol{S}\left(a\right)$.
\end{lem}

\begin{proof}
Since $a$ is generic we must have $\boldsymbol{S}\left(a\right),\boldsymbol{S}\left(b\right)\subseteq X$.
By Lemma \ref{lem:C2. Y_a is finite }, $\boldsymbol{S}\left(a\right)$
is a finite set. Denote $\boldsymbol{S}\left(a\right)=\left\{ a_{1},a_{2},\ldots a_{r}\right\} $
for some fixed ordering of $\boldsymbol{S}\left(a\right)$ in which
$a_{1}=a$. By the genericity of $a$, there is a $\emptyset$-definable
open-interval $J_{1}\subseteq N$ such that for all $x\in J_{1}$,
$|\boldsymbol{S}\left(x\right)|=r$.

Now, we can define $\emptyset$-definable functions $f_{i}:J_{1}\rightarrow N$,
$1\leq i\leq r$, such that for every $x\in J_{1}$, we have $\boldsymbol{S}\left(x\right)=\left\{ f_{1}\left(x\right),\ldots,f_{r}\left(x\right)\right\} $
and $f_{1}\left(x\right)=x$. So each $f_{i}\left(x\right)$ is a
shadow of $x$. By Lemma \ref{lem:C3. 2pts a s.t. b in Y_a}, the
$f_{i}$ cannot be constant on any open-interval. Hence, by the Monotonicity
Theorem for o-minimal structures~\cite{Dries} and the genericity
of $a$, there is a $\emptyset$-definable open-interval $J_{2}\subseteq J_{1}$,
$a\in J_{2}$, such that each $f_{i}$ is continuous (with respect
to the $\tau^{af}$-topology) and strictly monotone on $J_{2}$. Therefore,
$f_{i}|_{J_{2}}:J_{2}\rightarrow f_{i}\left(J_{2}\right)$ is a homeomorphism
(with respect to $\tau^{af}$) for all $1\leq i\leq r$.

Since $\boldsymbol{S}\left(a\right)$ is a finite set, there exists
an open-interval $J_{3}\subseteq J_{2}$, $a\in J_{3}$, such that
for all $1\leq i\neq j\leq r$, $f_{i}\left(J_{3}\right)\cap f_{j}\left(J_{3}\right)=\emptyset$.
Note that we might need additional parameters to define $J_{3}$,
but by Lemma \ref{lem:Generic open interval}, we can pick $J_{3}$
such that $a$ is still generic over its end points. To simplify notation
we absorb these additional parameters into the language, and thus
assume that $J_{3}$ is definable over $\emptyset$.

Recall that $b\in\boldsymbol{S}\left(a\right)$. We now prove a claim:\\
\\
\textbf{Claim.} For every open-interval $J\subseteq J_{3}$ such that
$a\in J$, there is $W\in\mathcal{B}_{b}$ such that $W\subseteq\bigcup_{i=1}^{r}f_{i}\left(J\right)$.$\,\,\,\,\,\,\,\,\,\,\,\,\,\,\,\,\,\,\,\,\,\,\,\,\,\,\,\,\,\,\,\,\,\,\,\,\,\,\,\,\,\,\,\,\,\,\,\,\,\,\,\,\,\,\,\,\,\,\,\,\,\,\,\,\,\,\,\,\,\,\,\,\,\,\,\,\,\,\,\,\,\,\,\,\,\,\,\,\,\,\,\,\,\,\,\,\,\,\,\,\,\,\,\,\,\,\,\,\,\,\,\,\,\,\,\,\,\,\,\,\,\,\,\,\,\,\,\,\,\,\,\,\,\,\,\,\,\,\,\,\,\,\,\,\,\,\,\,\,\,\,\,\,\,\,\,\,\,\,\,\,\,\,\,\,\,\,\,\,\,\,\,\,\,\,\,\,\,\,\,\,$\smallskip{}
\emph{Proof. }Assume towards contradiction that for some open-interval
$J\subseteq J_{3}$ such that $a\in J$,
\[
(*)\text{ for every }W\in\mathcal{B}_{b},\text{ we have }W\nsubseteq\bigcup_{i=1}^{r}f_{i}\left(J\right).
\]

By Lemma \ref{lem:Generic open interval}, we can replace $J$ by
an open-interval $J'\subseteq J$ with $a\in J'$, such that $b$
is still generic over the parameters defining $J'$. Note that we
still have that for every $W\in\mathcal{B}_{b}$, $W\nsubseteq\bigcup_{i=1}^{r}f_{i}\left(J'\right)$.
Thus we may assume that $b$ is generic over the parameters defining
$J$. We can now formulate $(*)$ as a definable property of $b$,
call it $\varphi\left(b\right)$. Since $b$ is generic over the parameters
defining $J$, there exists an open-interval $I\ni b$ such that $\varphi\left(y\right)$
is true for all $y\in I$.

By Lemma \ref{lem:GeneralProperty1 of Y_a}, there exists $U\in\mathcal{B}_{a}$
such that $U\subseteq\bigcup_{i=1}^{r}f_{i}\left(J\right)$. Clearly,
no $y\in U$ satisfies $\varphi\left(y\right)$, hence $U\cap I=\emptyset$.
It follows that $b\notin cl^{af}\left(U\right)$, contradicting the
fact that $b\in\boldsymbol{S}\left(a\right)$. $\boxempty$\\

Now we are ready to finish the proof of Lemma \ref{lem:Lemma Y_b in Y_a}.
By the Claim, given an open-interval $J\subseteq J_{3}$ with $a\in J$,
there is $W\in\mathcal{B}_{b}$ such that $W\subseteq\bigcup_{i=1}^{r}f_{i}\left(J\right)$.
Thus,
\[
cl^{af}\left(W\right)\subseteq\bigcup_{i=1}^{r}cl^{af}\left(f_{i}\left(J\right)\right).
\]
 Recall that $\boldsymbol{S}\left(b\right)=\bigcap_{V\in\mathcal{B}_{b}}cl^{af}\left(V\right)$,
and therefore
\[
\boldsymbol{S}\left(b\right)\subseteq\bigcup_{i=1}^{r}cl^{af}\left(f_{i}\left(J\right)\right),\text{ for any open-interval \ensuremath{J\ni a}.}
\]

By the continuity of the $f_{i}$, the intersection of all $\bigcup_{i=1}^{r}cl^{af}\left(f_{i}\left(J\right)\right)$,
as $J$ varies over all open-intervals containing $a$, is exactly
$\left\{ a_{1},a_{2},\ldots a_{r}\right\} $. Thus, $\boldsymbol{S}\left(b\right)\subseteq\left\{ a_{1},a_{2},\ldots a_{r}\right\} =\boldsymbol{S}\left(a\right)$.
\end{proof}
We give an example of $a\in X$ and of $b\in\boldsymbol{S}\left(a\right)$
that are not generic over $\emptyset$, for which the result of Lemma
\ref{lem:Lemma Y_b in Y_a} is not true:
\begin{example}
Let $X$ be an open-interval. We define a definable Hausdorff topology
on $X$, by describing small enough basic neighborhoods of three distinct
non-generic points $a,b,c\in X$: $U_{a}\in\mathcal{B}_{a}$ is of
the form $U_{a}=\left\{ a\right\} \cup\left(b-\epsilon,b\right)$,
$U_{b}\in\mathcal{B}_{b}$ is of the form $U_{b}=\left[b,b+\epsilon\right)\cup\left(c,c+\epsilon\right)$,
and $c$ is an isolated point. Every other $x\in X$ is Euclidean.
One can verify that $\boldsymbol{S}\left(a\right)=\left\{ a,b\right\} $
and $\boldsymbol{S}\left(b\right)=\left\{ b,c\right\} $, so $\boldsymbol{S}\left(b\right)\nsubseteq\boldsymbol{S}\left(a\right)$
and thus Lemma \ref{lem:Lemma Y_b in Y_a} is not true in this case.
\end{example}

\subsection{$\boldsymbol{\tau\subseteq\tau^{af}|_{X}}$ (every $\boldsymbol{\tau}$-open
set is also $\boldsymbol{\tau^{af}}$-open in $\boldsymbol{X}$)}

The purpose of this subsection is to analyze a special case, when
$\tau$ coarsens the affine topology on $X$. Namely, every $\tau$-open
set can be written as the intersection of $X$ and a definable $\tau^{af}$-open
subset of $M$.  We aim to prove the next theorem:
\begin{thm}
\label{Let--and-finite} Assume that $\tau\subseteq\tau^{af}|_{X}$,
that is, for all $x\in X$, $\mathcal{B}_{x}\preceq\mathcal{B}_{x}^{af}$.
Then there are at most finitely many points $a\in X$ such that $\mathcal{B}_{a}\nsim\mathcal{B}_{a}^{af}$.
Equivalently, there are at most finitely many $a\in X$ such that
$\mathcal{B}_{a}^{af}\npreceq\mathcal{B}_{a}$. \\

\emph{We first introduce:}
\end{thm}

\begin{defn}
Let $X=\left(s_{1},t_{1}\right)\sqcup\ldots\sqcup\left(s_{l},t_{l}\right)\sqcup F$
such that $F$ is finite be a definable subset of $M$. If $s_{i}\notin F$,
then each set of the form $\left(s_{i},r\right)$ for $s_{i}<r\leq t_{i}$,
is called a \emph{left generalized ray of $X$.} If $t_{i}\notin F$,
then each set of the form $\left(r,t_{i}\right)$ for $s_{i}\leq r<t_{i}$,
is called a \emph{right} \emph{generalized ray of $X$.} A left generalized
ray and a right generalized ray are both called \emph{generalized
rays. }
\end{defn}

For example, if $M=\mathbb{R}$ and $X=\left(3,\pi\right]=\left(3,\pi\right)\cup\left\{ \pi\right\} $,
then $\left(3,3.1\right)$ is a left generalized ray, but $\left(3.1,\pi\right)$
is not a generalized ray.
\begin{rem*}
Note that if $U$ is a definable subset of $X$ and $b\in cl^{af}\left(U\right)\setminus X$,
then $b$ is an endpoint of a generalized ray contained in $U$.
\end{rem*}
We will see that whenever $a\in X\setminus F$ has $\mathcal{B}_{a}\nsim\mathcal{B}_{a}^{af}$,
every neighborhood $U\in\mathcal{B}_{a}$ contains a generalized ray.
As a result, we conclude that there are at most $2l$ points $a\in X$
such that $\mathcal{B}_{a}\nsim\mathcal{B}_{a}^{af}$: Indeed, if
there were more than $2l$ such points, then by the pigeonhole principle,
two of those points would have intersecting generalized rays, in contradiction
to the fact that $\tau$ is Hausdorff.

Therefore, in order to prove Theorem \ref{Let--and-finite}, it is
sufficient to prove:
\begin{lem}
\label{Let--point-ray} Let $a\in X\setminus F$ be a point such that
$\mathcal{B}_{a}\nsim\mathcal{B}_{a}^{af}$. Then every neighborhood
$U\in\mathcal{B}_{a}$ contains a generalized ray.
\end{lem}

\begin{proof}
Since we are working under the assumption that every open set is also
$\tau^{af}$-open in $X$, $\mathcal{B}_{a}\nsim\mathcal{B}_{a}^{af}$
implies $\mathcal{B}_{a}\precneqq\mathcal{B}_{a}^{af}$. Therefore,
by Lemma \ref{lem: B_a < B_a^M --> Y_a =00005Cneq =00007Ba=00007D},
we have $\boldsymbol{S}\left(a\right)\neq\left\{ a\right\} $. Let
$b\in\boldsymbol{S}\left(a\right)$, $b\neq a$. We claim that $b\notin X$:
Every $U\in\mathcal{B}_{a}$ is also a $\tau^{af}$-open neighborhood
of $a$. So if we had $b\in X$, then since $b\in cl^{af}\left(U\right)$
for every $U\in\mathcal{B}_{a}$, $a$ and $b$ could not be separated,
contradicting the fact that $\tau$ is Hausdorff. Therefore $b\notin X$,
and as remarked above $U$ must contain a generalized ray.~$\square$

\medskip{}

This ends the proof of Theorem \ref{Let--and-finite}.
\end{proof}

\subsection{Almost $\boldsymbol{\tau\subseteq\tau^{af}|_{X}}$ }

The next technical lemma states two equivalent conditions that clarify
what we mean by ``almost $\tau\subseteq\tau^{af}|_{X}$''.
\begin{lem}
\label{lem:almost T^M|_X} Let $X\subseteq M$ be a definable set,
and let $\tau$ be a definable Hausdorff topology on $X$. Let $G\subseteq X$
be a finite set. The following are equivalent:

(1) For every $a\in X\setminus G$, $\mathcal{B}_{a}^{\tau}\preceq\mathcal{B}_{a}^{af}$.

(2) Every $\tau$-open subset of $X\setminus G$ is also $\tau^{af}$-open
in $X$ (that is, $\tau|_{X\setminus G}\subseteq\tau^{af}|_{X}$).
\end{lem}

\begin{proof}
$(1)\Rightarrow(2)$: Take $U'\in\tau|_{X\setminus G}$. That is,
there exists $U\in\tau$ such that $U'=U\cap\left(X\setminus G\right)$.
Since $\left(X\setminus G\right)\in\tau$, then $U'\in\tau$. So for
every $a\in U'$, there is a basic neighborhood $W_{a}\in\mathcal{B}_{a}^{\tau}$
such that $W_{a}\subseteq U'$. By (1), there is $V_{a}\in\mathcal{B}_{a}^{af}$
such that $V_{a}\subseteq W_{a}\subseteq U'\subseteq X$. Therefore,
$U'=\bigcup_{a\in U'}V_{a}$ is $\tau^{af}$-open in $X$, and hence
$U'\in\tau^{af}|_{X}$.

$(2)\Rightarrow(1)$: Fix $a\in X\setminus G$, and let $U\in\mathcal{B}_{a}^{\tau}$.
Since $G$ is $\tau$-closed then $U\setminus G\subseteq X\setminus G$
is $\tau$-open, that is, $U\setminus G\in\tau|_{X\setminus G}$.
By (2), we have $U\setminus G\in\tau^{af}|_{X}$. Because $a\in U\setminus G$,
there is a basic neighborhood $V\in\mathcal{B}_{a}^{af}$ such that
$V\subseteq\left(U\setminus G\right)\subseteq U$. Thus, $\mathcal{B}_{a}^{\tau}\preceq\mathcal{B}_{a}^{af}$.
\end{proof}
We proceed with some general lemmas and a theorem.
\begin{lem}
\label{lem:tau-open if and only if  <-open} Assume that $X$ is a subset of $M$ and that there is a finite set $G\subseteq X$ such that
on $X\setminus G$, every $\tau$-open set is $\tau^{af}$-open in
$X$.

Then there exists a definable set $X'\subseteq M$ and a definable
topology $\tau'$ on $X'$ such that $\left(X,\tau\right)$ is definably
homeomorphic to $\left(X',\tau'\right)$, and on each open-interval
$I'\subseteq X'$, a subset of $I'$ is $\tau'$-open if and only if  it is $\tau^{af}$-open.
\end{lem}

\begin{proof}
Denote $X=\left(s_{1},t_{1}\right)\sqcup\ldots\sqcup\left(s_{l},t_{l}\right)\sqcup F$
where $F$ is finite and $s_{i},t_{i}\in M$. Since on $X\setminus G$
every $\tau$-open set is $\tau^{af}$-open in $X$, by applying Theorem
\ref{Let--and-finite} to $\left(X\setminus G,\tau|_{X\setminus G}\right)$
we get that the set
\[
A:=\left\{ a\in X\setminus G:\mathcal{B}_{a}\nsim\mathcal{B}_{a}^{af}\right\} =\left\{ a\in X\setminus G:\mathcal{B}_{a}\precneqq\mathcal{B}_{a}^{af}\right\}
\]
is finite.

Denote $H:=F\cup G\cup A$, and fix the obvious cell decomposition
of $M$ compatible with $\left\{ \left(s_{1},t_{1}\right),\ldots,\left(s_{l},t_{l}\right),H\right\} $.
Define $X'$ as follows: Leave each $1$-cell as it is, and map $H$
to a disjoint set $H'$ of $\tau^{af}$-isolated points. So $X'=\left(q_{1},r_{1}\right)\sqcup\ldots\sqcup\left(q_{k},r_{k}\right)\sqcup H'$
for a finite $k\geq l$ and a finite set of points $H':=f\left(H\right)=f\left(F\cup G\cup A\right)$.
This gives us a definable bijection $f:X\rightarrow X'$.

Define the topology on $X'$ to be the obvious topology induced by
$\tau$ and $f$, that is, $\tau':=\left\{ f\left(U\right):U\in\tau\right\} $.
Thus, $f:\left(X,\tau\right)\longrightarrow\left(X',\tau'\right)$
is by definition a definable homeomorphism. Therefore, for every subset
$U'\subseteq\left(q_{i},r_{i}\right)$, $U$' is $\tau'$-open if and only if 
$U'$ is $\tau^{af}$-open.
\end{proof}
In the process of proving Theorem \ref{thm: TFAE affine} below, we
move from a definable topology on a one-dimensional $X\subseteq M^{n}$
to a definably homeomorphic topology on $X'\subseteq M$. While some
properties are obviously invariant under definable homeomorphism,
others might depend on the embedding of $X$ in $M^{n}$. We thus
first need:
\begin{lem}
\label{lem: Kobi's} Let $X\subseteq M^{n}$ and $X'\subseteq M^{k}$
be definable one-dimensional sets. If $f:X\to X'$ is a definable
bijection, then there is a finite set $G\subseteq X$ such that for
all $a\in X\setminus G$, the family of sets
\[
f(\mathcal{B}_{a}^{af}(X))=\left\{ f(U\cap X):U\text{ is }\tau^{af}\text{-open in \ensuremath{M^{n}}},a\in U\right\}
\]
forms a basis to the neighborhoods of $f(a)$ in the affine topology
$\tau^{af}$ on $X'$.
\end{lem}

\begin{proof}
By basic properties of definable functions in o-minimal structures,
there is a finite set $G\subseteq X$ such that $f:X\setminus G\to X'\setminus f(G)$
is a definable homeomorphism, with respect to the affine topology
on both $X\setminus G$ and $X'\setminus f(G)$. The result follows.
\end{proof}

\section{The main theorems}

\begin{thm}
\label{thm: TFAE affine} Let $X\subseteq M^{n}$ be a definable bounded set, $\dim X=1$, and
let $\tau$ be a definable Hausdorff topology on $X$. Then the following
are equivalent:
\end{thm}

\begin{enumerate}
\item \label{enu: homeomorphic affine}$\left(X,\tau\right)$ is definably
homeomorphic to a definable set with its affine topology.
\item \label{enu: subsets components}Every definable subset of $X$ has
a finite number of definably connected components, with respect to
$\tau$.
\item \label{enu: coarsens affine}For all but finitely many $x\in X$,
$\mathcal{B}_{x}\preceq\mathcal{B}_{x}^{af}$.
\item \label{enu: cofinite affine}There is a finite set $G\subseteq X$
such that on $X\setminus G$ every $\tau$-open set is $\tau^{af}$-open
in $X$.
\end{enumerate}
\begin{proof}
We observe first that if $f:(X,\tau)\longrightarrow(X',\tau')$ is
a definable homeomorphism, then for every $a\in X$, $f$ sends the
basis of $\tau$-neighborhoods $\mathcal{B}_{a}$ to a basis of $\tau'$-neighborhoods
of $f(a)\in X'$. By Lemma \ref{lem: Kobi's}, there exists a finite set
$G\subseteq X$ such that for all $a\in X\setminus G$, $f(\mathcal{B}_{a}^{af}(X))=\mathcal{B}_{f(a)}^{af}(X')$.
It follows that for all $a\in X\setminus G$,
\[
\mathcal{B}_{a}\preceq\mathcal{B}_{a}^{af}(X)\Longleftrightarrow\mathcal{B}_{f(a)}\preceq\mathcal{B}_{f(a)}^{af}(X').
\]

Thus, property (\ref{enu: coarsens affine}) holds for $X$ if and
only if it holds for $X'$. By the same lemma, (\ref{enu: cofinite affine})
holds for $X$ if and only if it holds for $X'$. Properties (\ref{enu: homeomorphic affine}),
(\ref{enu: subsets components}) are clearly invariant under definable
homeomorphisms.

The above discussion, together with Lemma \ref{lem:tau-open if and only if  <-open},
allows us to assume that $X$ is a bounded subset of $M$.

 $(1)\Rightarrow(2)$: If $\left(X,\tau\right)$ is definably homeomorphic
to a definable set with its affine topology, then by o-minimality,
every definable subset of $X$ has a finite number of definably connected
components.

$(2)\Rightarrow(3)$: Assume towards contradiction that there is an
infinite definable set of points $A\subseteq X$ such that $\mathcal{B}_{a}\mathcal{\npreceq B}_{a}^{af}$
for all $a\in A$. By Corollary \ref{cor: a locally iso=00005Cleft=00005Cright},
every $a\in A$ is either locally isolated or locally right-closed
or locally left-closed\emph{.} Notice that these properties are all
definable properties of $a$.

If there are infinitely many locally isolated points in $A$, then
the set \linebreak{}
$\left\{ a\in A:a\text{ is locally isolated}\right\} $ is a definable
infinite set, so contains an interval. Notice that for every locally
isolated point $a\in A$ and small enough $U_{a}\in\mathcal{B}_{a}$
there exists an open-interval $I\ni a$ such that $U_{a}\cap I=\left\{ a\right\} $.
Fix $a_{0}\in A$ generic over $\emptyset$ and an open-interval $I_{0}\ni a_{0}$
such that $U_{a_{0}}\cap I_{0}=\left\{ a_{0}\right\} $. Now, similarly
to the proof of Lemma \ref{lem:Lemma Y_b in Y_a}: We can fix an open-interval
$J_{3}\subseteq I_{0}\cap A$ of locally isolated points such that
for every  $a\in J_{3}$, for every small enough $U_{a}\in\mathcal{B}_{a}$,
we have $U_{a}\cap J_{3}=\left\{ a\right\} $.  Therefore, $J_{3}$
is a definable infinite set that is totally disconnected, contradicting
(2).

If there are infinitely many locally right-closed points in $A$,
then the set \linebreak{}
$\left\{ a\in A:a\text{ is locally right-closed}\right\} $ is a definable
infinite set. Similarly to the above, there exists an open-interval
$J$ such that every $a\in J$ is locally right-closed, and we can
obtain such a $J$ for which for every $a\in J$ there is $U_{a}\in\mathcal{B}_{a}$
such that $U_{a}\cap J=\left(a',a\right]$. Therefore, once again
$J$ is a definable infinite set for which the only definably connected
sets are singletons, contradicting (2). We treat similarly the remaining
case.

Notice that Lemma \ref{lem:Lemma Y_b in Y_a} is carried out in an
elementary extension $\mathcal{N}$ of $\mathcal{M}$. However, the
existence of an interval $J_{3}$ with all of these properties, is
easily seen to be a first-order property of the structure. Thus, after
possibly quantifying over parameters, we obtain the existence of such
an interval in the structure $\mathcal{M}$ in which we are working,
and obtain a contradiction in $\mathcal{M}$. Thus, we showed that
the set $A\subseteq X$ of all points $a\in X$ such that $\mathcal{B}_{a}\mathcal{\npreceq B}_{a}^{af}$
must be a finite set.

$(3)\Rightarrow(4)$: Assume that for all but finitely many $x\in X$
we have $\mathcal{B}_{x}\preceq\mathcal{B}_{x}^{af}$, and denote
$G=\left\{ x\in X:\mathcal{B}_{x}\npreceq\mathcal{B}_{x}^{af}\right\} $.
Thus, for the finite set $G\subseteq X$ we have, by Lemma \ref{lem:almost T^M|_X},
that on $X\setminus G$, every $\tau$-open set is $\tau^{af}$-open
in $X$.

$(4)\Rightarrow(1)$: We give a direct proof by showing how to embed
$\left(X,\tau\right)$ in $M^{3}$. By Lemma \ref{lem:tau-open if and only if  <-open},
we assume that $X\subseteq M$ is a finite union of disjoint open-intervals
and $\tau^{af}$-isolated points:
\[
X=\left(q_{1},r_{1}\right)\sqcup\ldots\sqcup\left(q_{k},r_{k}\right)\sqcup H,
\]
such that for every $U\subseteq\left(q_{i},r_{i}\right)$, $U$ is
$\tau$-open if and only if  $U$ is $\tau^{af}$-open. 

Because the only definable bijections at hand are those given by the underlying group operation $+$,
we need to be slightly careful in our construction of the embedding. We first identify $H=\{h_1,\ldots, h_n\}$
with a finite subset of points in $M^3$ of the form $\langle p_i,0,0\rangle $, with $p_1<p_2<\cdots<p_n$ in $M$, such that for all $i_1,i_2=1,\ldots, n$, and $j=1,\ldots, k$, 
$2k|p_{i_1}-p_{i_2}|<r_j-q_j$.

We would like to understand the $\tau$-neighborhoods of points in
$H$. Consider $h\in H$ and $U\in\mathcal{B}_{h}$. Since $\tau$
is Hausdorff, the finite set $H$ is $\tau$-closed, and thus $U\setminus H$
is a $\tau$-open set. So by our assumption, $U\setminus H$ is also
$\tau^{af}$-open. Thus, every small enough $U\in\mathcal{B}_{h}$
is of the form $U=V\sqcup\left\{ h\right\} $ for a certain $\tau^{af}$-open
$V\subseteq U\setminus H$.

If $h$ is not $\tau$-isolated, then similarly to the proof of Lemma
\ref{Let--point-ray}, up to equivalence of bases every small enough
$U\in\mathcal{B}_{h}$ is a finite union of generalized rays and the
singleton $\left\{ h\right\} $ itself:
\[
U=\left(\bigcup_{1\leq j\leq k}\left(q_{i_{j}},q_{i_{j}}+\epsilon\right)\right)\cup\left(\bigcup_{1\leq j\leq k}\left(r_{i_{j}}-\epsilon,r_{i_{j}}\right)\right)\sqcup\left\{ h\right\} .
\]

For every $i\in\left\{ 1,\ldots,k\right\} $, consider $\left(q_{i},r_{i}\right)$
and its two generalized rays. Assume first that for every $h\in H$
there exists a neighborhood $W\in\mathcal{B}_{h}$ that does not intersect
$\left(q_{i},r_{i}\right)$. In this case we identify  $\left(q_{i},r_{i}\right)$
with an  interval on the x-axis in $M^{3}$ of the same length, whose affine closure does not
intersect $H$.

Assume now that there exists $h'\in H$ such that every neighborhood
$W\in\mathcal{B}_{h^{'}}$ intersects $\left(q_{i},r_{i}\right)$.
As we pointed out above, it follows that every $W\in\mathcal{B}_{h^{'}}$
contains a generalized ray, say a left generalized ray in $\left(q_{i},r_{i}\right)$.
In this case we definably identify $\left(q_{i},r_{i}\right)$ with
a piecewise-linear curve $C_{i}$ in $M^{3}$ such that $\left(h',0,0\right)$ is the
endpoint of $C_{i}$ which corresponds to $q_{i}$. Note that since
$\tau$ is Hausdorff, if such $h'$ exists then it is unique. If there
is also $h''\in H$ such that every neighborhood $W\in\mathcal{B}_{h''}$
contains a right generalized ray of $\left(q_{i},r_{i}\right)$, then
we choose the curve $C_{i}$ such that its other endpoint is $\left(h'',0,0\right)$.
We may need to stretch, shrink or twist $C_{i}$ so it fits properly
in $M^{3}$, without intersecting any other point of $H$ and any
other image of an interval $\left(q_{j},r_{j}\right)$. All of the
above can be done definably in $\mathcal{M}$. This is possible since
both the set $H$ and the number $k$ are finite, and we chose to $p_i$'s to be sufficiently close to each other. If it happens to
be that $h'=h''$, then in $M^{3}$, both sides of $C_{i}$ will be
attached to $\left(h',0,0\right)$, closing a piecewise-linear loop. It may also happen
that we have to attach both sides of another curve $C_{j}$ to this
same $\left(h',0,0\right)$, and in this case we obtain several loops
attached to the same point $\left(h',0,0\right)$.

It is straightforward that by doing the above we get a definable embedding
$f:\left(X,\tau\right)\rightarrow\left(M^{3},\tau^{af}\right)$, which
is a definable homeomorphism when restricted to its image. Therefore,
the proof of this direction is complete.


This ends the proof of Theorem \ref{thm: TFAE affine}.
\end{proof}

\subsection{Main theorem}

Our goal is to prove Theorem \ref{thm: Condition =0000235 affine},
which yields an additional equivalent condition to the ones in Theorem
\ref{thm: TFAE affine} for when a definable topology is definably
homeomorphic to an affine one. Note that unlike condition (2) of Theorem
\ref{thm: TFAE affine}, our new condition (2) of Theorem \ref{thm: Condition =0000235 affine}
will only require $X$ itself to have finitely many definably connected
components. On the way to proving the theorem we shall gain a better
understanding of general definable one dimensional Hausdorff topologies.
\begin{thm}
\label{thm: regular+connected implies affine} Let $X\subseteq M^{n}$
be a definable bounded set, $\dim X=1$, and let $\tau$ be a definable Hausdorff, regular
topology on $X$. If $\left(X,\tau\right)$ is definably connected,
then $\left(X,\tau\right)$ is definably homeomorphic to a definable
set with its affine topology.
\end{thm}

\begin{proof}
As before, we may assume that $X$ is a subset of $M$ of
the form $X=\left(s_{1},t_{1}\right)\sqcup\ldots\sqcup\left(s_{l},t_{l}\right)\sqcup F$
with $F$ finite. We first prove this theorem in our sufficiently
saturated elementary extension $\mathcal{N}$ of $\mathcal{M}$, and
afterwards we explain why it holds also for our original $\mathcal{M}$.
Note that the Hausdorffness and regularity of $X$ can be formulated
in a first-order way, thus $\left(X\left(N\right),\tau\left(N\right)\right)$
is also Hausdorff and regular.
Let us see that $X\left(N\right)$ is definably connected: Assume
towards contradiction that $X\left(N\right)$ is not definably connected.
Let $Z=\varphi\left(N,\bar{c}\right)$ be a definable non-trivial
clopen subset of $X\left(N\right)$. It is easy to see that there
is a formula $\psi\left(\bar{y}\right)$, $|\bar{y}|=|\bar{c}|$,
such that every element $\bar{c}'\in N$ satisfies $\psi$ if and only if  $\varphi\left(N,\bar{c}'\right)$
is a non-trivial clopen subset of $X\left(N\right)$. Since $\mathcal{N}\vDash\exists\bar{y\:}\psi\negthinspace\left(\bar{y}\right)$,
also $\mathcal{M}\vDash\exists\bar{y}\:\psi\negthinspace\left(\bar{y}\right)$.
So for some $\bar{d}\subseteq  M$, $\varphi\left(M,\bar{d}\right)$ is a
non-trivial clopen subset of $X\left(M\right)$, and this is a contradiction.
Therefore, $X\left(N\right)$ must be definably connected.

\smallskip{}

We begin with a claim: \\
\\
\textbf{Claim 1.} There are at most finitely many locally isolated
points in $X$. $\,\,\,\,\,\,\,\,\,\,\,\,\,\,\,\,\,\,\,\,\,\,$\smallskip{}
\emph{Proof. }Assume towards contradiction that there is $a\in X$
generic over $\emptyset$, which is locally isolated. Let $U\in\mathcal{B}_{a}$
and $I\ni a$ be an open-interval such that $U\cap I=\left\{ a\right\} $.
As in the proof of Lemma \ref{lem:Lemma Y_b in Y_a}, we may assume
that there are definable continuous strictly monotone functions $f_{1},\ldots,f_{r}:I\rightarrow X$,
such that for all $x\in I$, $\boldsymbol{S}\left(x\right)=\left\{ f_{1}\left(x\right),\ldots,f_{r}\left(x\right)\right\} $
with $f_{1}\left(x\right)=x$, and $f_{i}\left(I\right)\cap f_{j}\left(I\right)=\emptyset$
for $1\leq i\neq j\leq r$.

Since we assume that $\left(X,\tau\right)$ is definably connected,
$a$ cannot be isolated. Thus, by Lemma \ref{lem: a locally isolated implies Y_a > =00007Ba=00007D}
we conclude that $\boldsymbol{S}\left(a\right)\supsetneqq\left\{ a\right\} $,
hence $r\geq2$. Let $b=f_{2}\left(a\right)\in\boldsymbol{S}\left(a\right)$
for $f_{2}$ as in Lemma \ref{lem:Lemma Y_b in Y_a}.

We show that there is no $W\in\mathcal{B}_{a}$ such that $cl\left(W\right)\subseteq U$:
Because $\boldsymbol{S}\left(a\right)\cap f_{2}\left(I\right)=\left\{ b\right\} $,
for every $W\in\mathcal{B}_{a}$ there must be some interval of the
form $\left(b',b\right)$ or $\left(b,b''\right)$ that is contained
in $W\cap f_{2}\left(I\right)$. Without loss of generality, $\left(b',b\right)\subseteq f_{2}\left(I\right)$.
By Lemma \ref{lem:newC cl(I)},
\[
cl\left(\left(b',b\right)\right)\supseteq\left\{ x\in X:\boldsymbol{S}\left(x\right)\cap\left(b',b\right)\neq\emptyset\right\} .
\]
By the definition of $f_{2}$, we also have
\[
\left\{ x\in X:\boldsymbol{S}\left(x\right)\cap\left(b',b\right)\neq\emptyset\right\} \supseteq f_{2}^{-1}\left(\left(b',b\right)\right).
\]

It follows that $cl\left(\left(b',b\right)\right)$ contains an infinite
subset of $I$, but $U\cap I=\left\{ a\right\} $, so $cl\left(W\right)$
cannot be contained in $U$ for $W\in\mathcal{B}_{a}$. That is, $\tau$
is not regular, and this is a contradiction. $\boxempty$ \textbf{}\\
\textbf{}\\
\textbf{Claim 2.} There are at most finitely many $x\in X$ such that
$\mathcal{B}_{x}\sim\mathcal{B}_{x}^{^{\left[\,\,\right)}}$ or $\mathcal{B}_{x}\sim\mathcal{B}_{x}^{^{\left(\,\,\right]}}$.\smallskip{}
\emph{Proof.} Assume towards contradiction that $J\subseteq X$ is
an open-interval such that for every $x\in J$, $\mathcal{B}_{x}\sim\mathcal{B}_{x}^{^{\left[\,\,\right)}}$
(without loss of generality). Thus we can assume that for every $x\in J$,
we have $\mathcal{B}_{x}=\mathcal{B}_{x}^{^{\left[\,\,\right)}}$.
Notice that although $J$ is an open set and each interval $\left[c,d\right)\subseteq J$
is open as well, we can not conclude immediately that $\left[c,d\right)$
is also closed, because we do not know that $X\setminus\left[c,d\right)$
is open. For this, we must use the regularity of $\tau$.

For every $x_{0}\in J$ generic over $\emptyset$, let $U:=\left[x_{0},z_{0}\right)\in\mathcal{B}_{x_{0}}$
be such that $U\subseteq J$. By the regularity of $\tau$ there exists
$W=\left[x_{0},y_{0}\right)\in\mathcal{B}_{x_{0}}$ such that $cl\left(W\right)\subseteq U$.
Note that since $cl\left(W\right)\setminus W\subseteq U\subseteq J$
and for every $x\in J$ we have $\mathcal{B}_{x}=\mathcal{B}_{x}^{^{\left[\,\,\right)}}$,
we must have $cl\left(W\right)=W=\left[x_{0},y_{0}\right)$ (because
every $a\in U\setminus\left[x_{0},y_{0}\right)$ has an open neighborhood
disjoint from $\left[x_{0},y_{0}\right)$ ). Therefore, $cl\left(W\right)$
is also open in $\tau$, and hence it is clopen. This is a contradiction
to $\left(X,\tau\right)$ being definably connected. $\boxempty$
\\

We proceed with our proof of Theorem \ref{thm: regular+connected implies affine}.
We assume that $\left(X,\tau\right)$ is not definably homeomorphic
to any definable set with its affine topology, and we show that $X$
contains a definable clopen set. In fact, given Claim 1 and Claim
2 we shall not make further use of regularity. \\

Using   Theorem \ref{thm: TFAE affine} (3), there is a point $a\in X$ generic
over $\emptyset$ such that $\mathcal{B}_{a}\mathcal{\npreceq B}_{a}^{af}$.
By Corollary \ref{cor: a locally iso=00005Cleft=00005Cright}, $a$
is locally isolated or locally right-closed or locally left-closed\emph{.}
By Claim 1, $a$ is locally right-closed or left-closed\emph{.}

If $\boldsymbol{S}\left(a\right)=\left\{ a\right\} $, then we must
have either $\mathcal{B}_{a}\sim\mathcal{B}_{a}^{^{\left[\,\,\right)}}$
or $\mathcal{B}_{a}\sim\mathcal{B}_{a}^{^{\left(\,\,\right]}}$ by
Lemma~\ref{lem:GeneralProperty2 of Y_a}. Both cases are not possible
due to Claim~2. Thus, we assume from now on that for any generic
$x\in X$ such that $\mathcal{B}_{x}\mathcal{\npreceq B}_{x}^{af}$,
the set $\boldsymbol{S}\left(x\right)$ properly contains $\left\{ x\right\} $.
\\
\\
\textbf{Claim 3. }$|\boldsymbol{S}\left(a\right)|=2$, and if $b\in\boldsymbol{S}\left(a\right)$
then $\boldsymbol{S}\left(a\right)=\boldsymbol{S}\left(b\right)=\left\{ a,b\right\} $.$\,\,\,\,\,\,\,\,\,\,\,\,\,\,\,\,\,\,\,\,\,\,\,\,\,\,\,\,\,\,\,\,\,\,\,\,\,\,\,\,\,\,\,\,\,\,\,\,\,\,$\smallskip{}
\emph{Proof. }Since $\boldsymbol{S}\left(a\right)$ properly contains
$\left\{ a\right\} $, we have $|\boldsymbol{S}\left(a\right)|\geq2$.
Let $b\in\boldsymbol{S}\left(a\right)$, $b\neq a$. Note that since
$a$ is generic over $\emptyset$ then by Lemma \ref{lem: a generic implies b generic}
so is $b$. By Lemma \ref{lem: (*)}, $\mathcal{B}_{b}\mathcal{\npreceq B}_{b}^{af}$,
so $\boldsymbol{S}\left(b\right)\supsetneqq\left\{ b\right\} $. Since
$b$ is generic, it follows from Lemma \ref{lem:Lemma Y_b in Y_a}
that $\boldsymbol{S}\left(b\right)\subseteq\boldsymbol{S}\left(a\right)$.

Assume towards contradiction that $\boldsymbol{S}\left(b\right)\ne\left\{ a,b\right\} $.
Hence, there is $c\in\boldsymbol{S}\left(b\right)$ (so also in $\boldsymbol{S}\left(a\right)\,$),
$c\neq a,b\,$. By Lemma \ref{lem: a generic implies b generic},
$c$ is generic over $\emptyset$, so by Lemma \ref{lem:C3. 2pts a s.t. b in Y_a}
it must also be locally isolated, contradicting Claim 1. Therefore,
it must be that $\boldsymbol{S}\left(b\right)=\left\{ a,b\right\} $.

By replacing $a$ and $b$ in the above, we also get $\boldsymbol{S}\left(a\right)=\left\{ a,b\right\} $.~$\boxempty$
 \\

We say that a point $x\in X$ \emph{inhabits the left side} of a point
$y\in X$ if for every $U\in\mathcal{B}_{x}$ there exists $y'\in X$,
$y'<y$, such that $\left(y',y\right)\subseteq U$. We say that $x$
\emph{inhabits the right side} of $y\in X$ if for every $U\in\mathcal{B}_{x}$
there exists $y''\in X$, $y<y''$, such that $\left(y,y''\right)\subseteq U$.

We note several easy observations for $a$ that is generic over $\emptyset$,
not locally isolated, such that $\mathcal{B}_{a}\mathcal{\npreceq B}_{a}^{af}$:

(1) If $a$ inhabits the left side or the right side of $b$ then
$b\in\boldsymbol{S}\left(a\right)$.

(2) Conversely, if $b\in\boldsymbol{S}\left(a\right)$ then $a$ inhabits
the left side or the right side (or~both)

$\,\,\,\,\,\,\,\,\,\,$of $b$. (For the case $b=a$ we use here the
fact that $a$ is not locally isolated).

(3) $a$ cannot inhabit both sides of $b$. Indeed, since $b$ is
generic over $\emptyset$ then it is

$\,\,\,\,\,\,\,\,\,\,$not locally isolated by Claim 1, and since
$\tau$ is Hausdorff it must be possible

$\,\,\,\,\,\,\,\,\,\,$to separate between $a$ and $b$.

(4) $a$ inhabits the left side (the right side) of $b$ if and only if  $b\in\boldsymbol{S}\left(a\right)$
and $b$ is locally left-closed (locally right-closed). \\

By Claim 3, $\boldsymbol{S}\left(a\right)=\left\{ a,b\right\} =\boldsymbol{S}\left(b\right)$
for $b\neq a$, and from its proof we deduce that $a$ inhabits exactly
one side of $a$ and exactly one side of~$b$, and so does $b$.

As we have seen before, we can find an interval $J\ni a$ and definable
continuous and strictly monotone functions $f_{1},f_{2}:J\rightarrow X$
with $f_{1}\left(J\right)\cap f_{2}\left(J\right)=\emptyset$, such
that for every $x\in J$, $\boldsymbol{S}\left(x\right)=\left\{ f_{1}\left(x\right)=x,f_{2}\left(x\right)\right\} $.
Moreover, the genericity of $a$ also implies that we may choose $J$
such that all $x\in J$ are ``of the same form'' as $a$. Namely,
\\
\\
(i) Every $x\in J$ is locally left-closed or every $x\in J$ is locally
right-closed.$\,\,\,\,\,\,\,\,\,\,$\smallskip{}
(ii) Every $y\in f_{2}\left(J\right)$ is locally left-closed or every
$y\in f_{2}\left(J\right)$ is locally right-closed.\\

Without loss of generality, assume that every $x\in J$ is locally
left-closed and every $y\in f_{2}\left(J\right)$ is locally right-closed
(the other cases are treated similarly). By (4), $x$ inhabits the
right side of $y:=f_{2}\left(x\right)$ and $y$ inhabits the left
side of $x$. \\
\\
\textbf{Claim 4.} Under these assumptions, $f_{2}$ is strictly increasing.$\,\,\,\,\,\,\,\,\,\,\,\,\,\,\,\,\,\,\,\,\,\,\,\,\,\,\,\,\,\,\,\,\,\,\,\,\,\,\,\,\,\,\,\,\,\,\,\,\,$\smallskip{}
\emph{Proof.} Assume towards contradiction that $f_{2}$ is strictly
decreasing. That is, for every $c,d\in J$, if $c<d$ then $f_{2}\left(c\right)>f_{2}\left(d\right)$.
Fix $c\in J$ generic over $\emptyset$. By our assumption, $c$ is
locally left-closed, that is, for every small enough $U\in\mathcal{B}_{c}$
there exists an open-interval $I_{U}\ni c$ and a point $x>c$ such
that $U\cap I_{U}=\left[c,x\right)$.

By our assumption, $f_{2}\left(c\right)$ is locally right-closed.
So $f_{2}$ being strictly decreasing implies that for every $W\in\mathcal{B}_{f_{2}\left(c\right)}$
and $x>c$, we have
\[
W\cap f_{2}\left(\left[c,x\right)\right)=W\cap\left(f_{2}\left(x\right),f_{2}\left(c\right)\right]\neq\emptyset.
\]

Note that since $f_{2}$ is strictly decreasing and $f_{2}\left(c\right)\in\boldsymbol{S}\left(c\right)$
, there must be $y'<f_{2}\left(c\right)$ such that $\left(y',f_{2}\left(c\right)\right]\subseteq U$.
Thus, for every $U\in\mathcal{B}_{c}$, we must have $W\cap U\supseteq W\cap\left(y',f_{2}\left(c\right)\right]\neq\emptyset$.
This contradicts the fact that $\tau$ is Hausdorff, and therefore
$f_{2}$ must be strictly increasing.~$\boxempty$  \\

Recall that for every $x\in J$ we have $\boldsymbol{S}\left(x\right)=\left\{ x,y\right\} =\boldsymbol{S}\left(y\right)$,
and as we just showed $f_{2}$ is strictly increasing. By Lemma \ref{lem:GeneralProperty1 of Y_a},
we know that for every small enough $U\in\mathcal{B}_{x}$, $U\subseteq\left(J\cup f_{2}\left(J\right)\right)$,
and for every small enough $W\in\mathcal{B}_{y}$, $W\subseteq\left(J\cup f_{2}\left(J\right)\right)$.
Therefore, under our assumptions we get that for every $x\in J$,
\[
\mathcal{B}_{x}\sim\left\{ \left[x,x''\right)\cup\left(y,y''\right):x''\in J,x<x'',y''=f_{2}\left(x''\right)\right\} ,
\]
 and for every $y\in f_{2}\left(J\right)$,
\[
\mathcal{B}_{y}\sim\left\{ \left(x',x\right)\cup\left(y',y\right]:x\in J,y'<y,x'=f_{2}^{-1}\left(y'\right)\right\} .
\]
 So by replacing the bases we can assume that
\[
\left(\ast\right)\begin{cases}
\mathcal{B}_{x}=\left\{ \left[x,x''\right)\cup\left(y,y''\right):x''\in J,x<x'',y''=f_{2}\left(x''\right)\right\} \\
\mathcal{B}_{y}=\left\{ \left(x',x\right)\cup\left(y',y\right]:y'\in f_{2}\left(J\right),y'<y,x'=f_{2}^{-1}\left(y'\right)\right\} .
\end{cases}
\]

For every $x\in J$ and $y=f_{2}\left(x\right)\in f_{2}\left(J\right)$,
we consider the definable families:
\[
\mathcal{B}_{x}^{f}:=\left\{ U\cap f_{2}\left(J\right):U\in\mathcal{B}_{x}\right\} \text{ and\,\, }\mathcal{B}_{y}^{f^{-1}}:=\left\{ W\cap J:W\in\mathcal{B}_{y}\right\} .
\]
Thus we have: \\
\\
(iii) $\mathcal{B}_{x}^{f}=\left\{ \left(y,y''\right):y''\in f_{2}\left(J\right),y<y''\right\} $.
$\,\,\,\,\,\,\,\,\,\,\,\,\,\,\,\,\,\,\,\,\,\,\,\,\,\,\,\,\,\,\,\,\,\,\,\,\,\,\,\,\,\,\,\,\,\,\,\,\,\,\,\,\,\,\,\,\,\,\,\,\,\,\,\,\,\,\,\,\,\,\,\,\,\,\,\,\,\,\,\,\,\,\,\,\,\,\,\,\,\,\,\,\,\,\,\,\,\,\,\,\,$\smallskip{}
(iv) $\mathcal{B}_{y}^{f^{-1}}=\left\{ \left(x',x\right):x'\in J,x'<x\right\} $.
\\

We are now ready to prove that $X$ is not definably connected. \\
\textbf{}\\
\textbf{Claim 5.} For $a''\in J$ with $a''>a$, the set $Z:=\left[a,a''\right)\cup\left(f_{2}\left(a\right),f_{2}\left(a''\right)\right]$
is clopen.$\,\,\,\,\,\,\,\,\,\,\,\,\,\,\,\,\,\,\,\,\,\,\,\,\,\,\,\,\,\,\,\,\,\,\,\,\,\,\,\,\,\,\,\,\,\,\,\,\,\,\,\,\,\,\,\,\,\,\,\,\,\,\,\,\,\,\,\,\,\,\,\,\,\,\,\,\,\,\,\,\,\,\,\,\,\,\,\,\,\,\,\,\,\,\,\,\,\,\,\,\,\,\,\,\,\,\,\,\,\,\,\,\,\,\,\,\,\,\,\,\,\,\,\,\,\,\,\,\,\,\,\,\,\,\,\,\,\,\,\,\,\,\,\,\,\,\,\,\,\,\,\,\,\,\,\,\,\,\,\,\,\,\,\,\,\,\,\,\,\,\,\,\,\,\,\,\,\,\,\,\,\,\,\,\,\,\,\,\,\,\,\,\,\,\,\,\,\,\,\,\,\,\,\,\,\,\,\,\,\,\,\,\,\,\,\,\,\,\,\,\,\,\,\,\,\,\,\,\,\,\,$\smallskip{}
\emph{Proof.} By $\left(\ast\right)$, $Z$ is open as the union
of two basic open sets. We explain why $Z$ is also closed: By general
properties of closure and since each singleton is closed, we have
\[
cl\left(Z\right)=\left\{ a\right\} \cup cl\left(\left(a,a''\right)\right)\cup cl\left(\left(f_{2}\left(a\right),f_{2}\left(a''\right)\right)\right)\cup\left\{ f_{2}\left(a''\right)\right\} .
\]
Thus, by Lemma \ref{lem:newC cl(I)} we deduce that
\begin{align*}
\left\{ a\right\} \cup\left\{ x\in\negthinspace X\negthinspace:\boldsymbol{S}\negthinspace\left(x\right)\cap\left(a,a''\right)\ne\emptyset\right\} \cup\left\{ x\in\negthinspace X\negthinspace:\boldsymbol{S}\negthinspace\left(x\right)\cap\left(f_{2}\left(a\right)\negthinspace,f_{2}\left(a''\right)\right)\ne\emptyset\right\} \cup\left\{ f_{2}\left(a''\right)\right\} \\
\subseteq cl\left(Z\right)\subseteq\,\,\,\,\,\,\,\,\,\,\,\,\,\,\,\,\,\,\,\,\,\,\,\,\,\,\,\,\,\,\,\,\,\,\,\,\,\,\,\,\,\,\,\,\,\,\,\,\,\,\,\,\,\,\,\,\,\,\,\,\,\,\,\,\,\,\,\,\,\,\,\,\,\,\,\,\,\,\,\,\,\,\,\,\,\,\,\,\,\,\,\,\,\,\,\\
\left\{ a\right\} \cup\left\{ x\in\negthinspace X\negthinspace:\boldsymbol{S}\negthinspace\left(x\right)\cap\left[a,a''\right]\ne\emptyset\right\} \cup\left\{ x\in\negthinspace X\negthinspace:\boldsymbol{S}\negthinspace\left(x\right)\cap\left[f_{2}\left(a\right)\negthinspace,f_{2}\left(a''\right)\right]\ne\emptyset\right\} \cup\left\{ f_{2}\left(a''\right)\right\} \negthinspace.
\end{align*}

The difference between the right hand side and the left hand side
is $\left\{ a'',f_{2}\left(a\right)\right\} $. Let us show that these
two points are not in $cl\left(Z\right)$: For $a''$, we know that
for $a'''\in J$ with $a'''>a''$, the set $\left[a'',a'''\right)\cup\left(f_{2}\left(a''\right),f_{2}\left(a'''\right)\right]$
is an open neighborhood of $a''$ which does not intersect $Z$. Thus,
$a''\notin cl\left(Z\right)$. Similarly, the point $f_{2}\left(a\right)$
has an open neighborhood of the form $\left[a',a\right)\cup\left(f_{2}\left(a'\right),f_{2}\left(a\right)\right]$,
which does not intersect $Z$. Thus, we have $f_{2}\left(a\right)\notin cl\left(Z\right)$.

Notice that $\boldsymbol{S}\left(x\right)\cap\left(a,a''\right)\ne\emptyset$
if and only if  $x\in\left(a,a''\right)$, and $\boldsymbol{S}\left(x\right)\cap\left(f_{2}\left(a\right),f_{2}\left(a''\right)\right)\ne\emptyset$
if and only if  $f_{2}\left(x\right)\in\left(f_{2}\left(a\right),f_{2}\left(a''\right)\right)$.
Therefore, we conclude that
\[
cl\left(Z\right)=\left\{ a\right\} \cup\left(a,a''\right)\cup\left(f_{2}\left(a\right),f_{2}\left(a''\right)\right)\cup\left\{ f_{2}\left(a''\right)\right\} =\left[a,a''\right)\cup\left(f_{2}\left(a\right),f_{2}\left(a''\right)\right]=Z,
\]
 hence we proved that $Z$ is clopen.~$\boxempty$ \\

By Claim 5, $X$ contains the definable clopen set $Z$. That is,
$\left(X,\tau\right)$ is not definably connected. Hence, we proved
Theorem \ref{thm: regular+connected implies affine} for our sufficiently
saturated $\mathcal{N}$. \\
\\
\textbf{Claim 6.} Theorem \ref{thm: regular+connected implies affine}
holds also in $\mathcal{M}$. $\,\,\,\,\,\,\,\,\,\,\,\,\,\,\,\,\,\,\,\,\,\,\,\,\,\,\,\,\,\,\,\,\,\,\,\,\,\,\,\,\,\,\,\,\,\,\,\,\,\,\,\,\,\,\,\,\,\,\,\,\,\,\,\,\,\,\,\,\,\,\,\,\,\,\,\,\,\,\,\,\,\,\,\,\,\,\,\,\,\,\,\,\,\,\,\,\,\,\,\,\,\,\,\,\,\,\,\,\,\,\,\,\,$\smallskip{}
\emph{Proof.} At the beginning of the proof of Theorem \ref{thm: regular+connected implies affine}
we saw that if $\left(X\left(M\right),\tau\left(M\right)\right)$
is Hausdorff, regular and definably connected, then so is $\left(X\left(N\right),\tau\left(N\right)\right)$.
Therefore, if $\mathcal{M}$ satisfies the conditions of Theorem \ref{thm: regular+connected implies affine},
then $\mathcal{N}$ satisfies them as well.

In this case, we get from the theorem that there exist a definable
bijection $f_{\bar{c}}:X\left(N\right)\rightarrow S_{\bar{c}}$ where
$S_{\bar{c}}\subseteq N^{k}$ for some $k$, such that $f_{\bar{c}}\,,S_{\bar{c}}$
are definable over parameters $\bar{c}\sub N$, and $f_{\bar{c}}$
is a homeomorphism of $\left(X\left(N\right),\tau\left(N\right)\right)$
and $S_{\bar{c}}$ with the affine topology. We can now write a formula
$\psi\left(\bar{y}\right)$, $|\bar{y}|=|\bar{c}|$, such that for
every $\bar{c}'\in N$,
\begin{align*}
\mathcal{N} & \vDash\psi\left(\bar{c}'\right)\Longleftrightarrow f_{\bar{c}'}\text{ is a (definable) homeomorphism of }\left(X\left(N\right),\tau\left(N\right)\right)\\
 & \text{ \,\,\,\,\,\,\,\,\,\,\,\,\,\,\,\,\,\,\,\,\,\,\,\,\,\,\,\,\,\,\,\,\,\,\,\,\,and a (definable) }S_{\bar{c}'}\subseteq N^{k}\text{ with the affine topology}.
\end{align*}

Since $\mathcal{N}\vDash\exists\bar{y\;}\psi\left(\bar{y}\right)$,
then so does $\mathcal{M}$, and hence there exists $\bar{d}\sub M$
such that $f_{\bar{d}}:X\left(M\right)\rightarrow S_{\bar{d}}\left(M\right)$
is the desired homeomorphism. That is, Theorem \ref{thm: regular+connected implies affine}
holds also for $\mathcal{M}$. This ends the proof of both Claim 6
and Theorem \ref{thm: regular+connected implies affine}.
\end{proof}
We can now add to Theorem \ref{thm: TFAE affine} another equivalent
condition:
\begin{thm}
\label{thm: Condition =0000235 affine} Let $X\subseteq M^{n}$ be a definable bounded set, $\dim X=1$,
and let $\tau$ be a definable Hausdorff topology on $X$. Then the following
are equivalent:
\end{thm}

\begin{enumerate}
\item $\left(X,\tau\right)$ is definably homeomorphic to a definable set
with its affine topology.
\item $\tau$ is regular, and $\left(X,\tau\right)$ has finitely
many definably connected components.
\end{enumerate}
\begin{proof}
$(1)\Rightarrow(2)$: This follows from the basic theory of o-minimal
structures (as is discussed in \cite{Dries}).

$(2)\Rightarrow(1)$: By assumption, $X$ is a disjoint union of definable
sets $X_{1},\ldots,X_{m}$, each open (hence closed) in $X$, and
definably connected with respect to $\tau$. Since $X$ is regular
so is each $X_{i}$ (with the induced $\tau$ topology). By Theorem
\ref{thm: regular+connected implies affine}, each $X_{i}$ is definably
homeomorphic to some $Y_{i}\subseteq M^{k_{i}}$ with its affine topology.
Let $k:=max_{1\leq i\leq m}k_{i}+1$, and embed each $Y_{i}$ in $M^{k}$.
Furthermore, we may choose the sets $Y_{i}$ such that
\[
cl^{af}\left(Y_{i}\right)\cap cl^{af}\left(Y_{j}\right)=\emptyset
\]
 for $i\neq j$. It follows that $X$ is definably homeomorphic to
$Y:=\bigsqcup_{i=1}^{m}Y_{i}$.
\end{proof}
A combination of Theorem \ref{thm: TFAE affine} and Theorem \ref{thm: Condition =0000235 affine}
gives us the Main theorem, as stated in the introduction:\\
\textbf{}\\
\textbf{Main theorem.}\emph{ Let $\mathcal{M}$ be an o-minimal expansion
of a linearly ordered group. Let $X\subseteq M^{n}$ be a definable bounded set
with $\dim X=1$, and let $\tau$ be a definable Hausdorff topology
on $X$. Then the following are equivalent: }
\begin{enumerate}
\item \emph{$\left(X,\tau\right)$ is definably homeomorphic to a definable
subset of $M^{k}$ for some $k$, with its affine topology. }
\item \emph{There is a finite set $G\subseteq X$ such that every $\tau$-open
subset of $X\setminus G$ is open with respect to the affine topology
on $X\setminus G$. }
\item \emph{Every definable subset of $X$ has finitely many definably connected
components, with respect to $\tau$.}
\item \emph{$\tau$ is regular and $X$ has finitely many definably connected
components, with respect to $\tau$.}\\
\end{enumerate}

We end with another theorem that is an immediate consequence of our
work towards Theorem \ref{thm: Condition =0000235 affine}:
\begin{thm}
Let $X\subseteq M^{n}$ be a definable bounded set with $\dim X=1$, and let
$\tau$ be a Hausdorff topology on $X$. Assume that $X$ has finitely
many locally isolated points, and finitely many points $x$ such that
$\mathcal{B}_{x}\sim\mathcal{B}_{x}^{^{\left[\,\,\right)}}$ or $\mathcal{B}_{x}\sim\mathcal{B}_{x}^{^{\left(\,\,\right]}}$.
If, in addition, $\left(X,\tau\right)$ has finitely many definably
connected components, then $\left(X,\tau\right)$ is definably homeomorphic
to a definable set with its affine topology.
\end{thm}

\begin{proof}
The proof follows from Theorem \ref{thm: Condition =0000235 affine},
since in the proof of Theorem \ref{thm: regular+connected implies affine}
(which leads to Theorem \ref{thm: Condition =0000235 affine}), we
only used regularity in Claim 1 and in Claim 2.
\end{proof}

\subsection{\label{subsec:Examples}Example }

The next example is of a definable Hausdorff topology that is definably
connected and not regular, thus it can not be definably homeomorphic
to a definable set with its affine topology. It demonstrates the necessity
of two different assumptions in two different principal theorems:

For Theorem \ref{thm: TFAE affine}, it shows that for the direction
$(2)\Rightarrow(1)$ it would have not been enough to only assume
that $\left(X,\tau\right)$ is definably connected. For Theorem \ref{thm: regular+connected implies affine},
it demonstrates that it is not enough to only assume that $\left(X,\tau\right)$
is Hausdorff and definably connected, and it is necessary to add the
assumption that $\tau$ is regular.
\begin{example}
Let $M\supseteq X=\left(0,1\right)\cup\left(1,2\right)\cup\left(2,3\right)$
be the union of three disjoint open-intervals. We define a definable
topology $\tau$ on $X$ via families of basic neighborhoods of points:
For $a\in\left(0,1\right)$, take
\[
\mathcal{B}_{a}:=\left\{\left\{ a\right\} \cup\left(a+1,b\right)\cup\left(c,a+2\right):a+1<b<2\,,2<c<a+2\right\} ,
\]
for $b\in\left(1,2\right)$, take
\[
\mathcal{B}_{b}:=\left\{ \left(b',b\right]:1<b'<b\right\} ,
\]
and for $c\in\left(2,3\right)$, take
\[
\mathcal{B}_{c}:=\left\{ \left[c,c''\right):c<c''<3\right\} .
\]

One can check that this topology is not regular. This fact, as well
as points of $\left(1,2\right)\cup\left(2,3\right)$ having half-open-intervals
neighborhoods, guaranties that $\tau$ is not definably homeomorphic
to the affine topology $\tau^{af}$.

The topology $\tau$ is definably connected since the only definable
clopen subsets of $X$, are $\emptyset$ and $X$ itself. Nevertheless,
$X$ contains definable subsets that are totally definably disconnected
(for instance, the only connected components of the interval $\left(1,2\right)$
are its singletons).
\end{example}

\subsection{Some final comments}


\begin{enumerate}
\item Note that the Main Theorem  fails, as stated, without the assumption that $X$ is bounded: Let $\mathcal{M}=\left(\mathbb{R};<,+\right)$
and let $X\subseteq\mathbb{R}^{2}$ be the union of the line $\mathbb{R}\times\left\{ 0\right\} $
and two other points $p_{1},p_{2}\in\mathbb{R}^{2}$. We can endow
$X$ with a topology $\tau$ which identifies it with $\left\{ -\infty\right\} \cup\mathbb{R}\cup\left\{ +\infty\right\} $, with $p_1,p_2$ identified with $-\infty$, $+\infty$, respectively. It is easy to verify that clauses (2)-(4) of Theorem \ref{thm: TFAE affine} hold. Let us see that (1) fails. 
Indeed,  it is not hard to verify that $\left(X,\tau\right)$ is definably compact
(note that here we make an exception and use the term ``definably
compact'' with respect to the topology $\tau$). Therefore, if $\left(X,\tau\right)$
were definably homeomorphic to a definable set $X'\subseteq\mathbb{R}^{n}$
with its affine topology, then $X'$ would have to be bounded. However,
in $\mathcal{M}$ there is no definable bijection between bounded
and unbounded sets.

\item Naturally, the ultimate goal of this project is to understand definable topologies in arbitrary dimension.
Towards this goal some modifications are needed in the Main Theorem. For example, it is not hard to see that Clause (2)
cannot hold as such since the set of points $a\in X$ at which the $\tau$-topology is different than  the  affine topology
can be infinite (but probably of dimension strictly smaller than $\dim X$). At any case, we do not know if the equivalence of (1), (3) and (4) in that theorem is still true for arbitrary dimension.

\end{enumerate}

\end{document}